\documentclass[11pt,letterpaper]{amsart}
\usepackage[centertags]{amsmath}
\usepackage{amsfonts}
\usepackage{amssymb}
\usepackage{amsthm}
\usepackage{newlfont}
\usepackage{amscd}
\usepackage{mathrsfs}
\usepackage[all,cmtip]{xy}
\usepackage{tikz-cd}
\tikzcdset{arrow style=Latin Modern, row sep/normal=1.35em, column sep/normal=1.8em, cramped}
\usepackage{tikz}
\usepackage[pagebackref=false,colorlinks]{hyperref}

\definecolor{mycolor}{HTML}{F7F8E0}
\definecolor{myorange}{RGB}{245,156,74}
\definecolor{cadetgrey}{rgb}{0.57, 0.64, 0.69}
\definecolor{calpolypomonagreen}{rgb}{0.12, 0.3, 0.17}

\hypersetup{pdffitwindow=true,linkcolor=calpolypomonagreen,citecolor=calpolypomonagreen,urlcolor=calpolypomonagreen}
\usepackage{cite}
\usepackage{fancyhdr}
\usepackage{stmaryrd}

\usepackage[OT2,OT1]{fontenc}
\newcommand\cyr{%
\renewcommand\rmdefault{wncyr}%
\renewcommand\sfdefault{wncyss}%
\renewcommand\encodingdefault{OT2}%
\normalfont
\selectfont}
\DeclareTextFontCommand{\textcyr}{\cyr}

\usepackage[toc, page] {appendix}

\numberwithin{equation}{section}

\newtheorem{thm}{Theorem}[section]

\newtheorem{cor}[thm]{Corollary}
\newtheorem{lem}[thm]{Lemma}
\newtheorem{prop}[thm]{Proposition}

\newtheorem{assu}[thm]{Assumption}

\theoremstyle{definition}
\newtheorem{defn}[thm]{Definition}
\newtheorem{choice}[thm]{Choice}
\newtheorem{rem}[thm]{Remark}

\newtheorem{ques}[thm]{Question}

\newcommand{\stark}{\boldsymbol{\epsilon}}

\newcommand{\ks}{\boldsymbol{\kappa}}

\begin{document}
\title{On the adjoint Selmer groups of semi-stable elliptic curves and Flach's zeta elements}
\author{Chan-Ho Kim}
\address{Department of Mathematics and Institute of Pure and Applied Mathematics,
Jeonbuk National University,
567 Baekje-daero, Deokjin-gu, Jeonju, Jeollabuk-do 54896, Republic of Korea}
\email{chanho.math@gmail.com}
\thanks{Chan-Ho Kim was partially supported 
by a KIAS Individual Grant (SP054103) via the Center for Mathematical Challenges at Korea Institute for Advanced Study,
by the National Research Foundation of Korea(NRF) grant funded by the Korea government(MSIT, ME) (No. 2018R1C1B6007009, 2019R1A6A1A11051177, RS-2025-16067678, RS-2026-25593558), 
by research funds for newly appointed professors of Jeonbuk National University in 2024, and
by Global-Learning \& Academic research institution for Master’s$\cdot$Ph.D. Students, and Postdocs (LAMP) Program of the National Research Foundation of Korea (NRF) funded by the Ministry of Education (No. RS-2024-00443714).
}
\date{\today}
\maketitle
\begin{abstract} 
We explicitly construct the rank one primitive Stark (equivalently,  Kolyvagin) system extending a constant multiple of Flach's zeta elements for semi-stable elliptic curves.
As its arithmetic applications, we obtain the equivalence between the $p$-indivisibility of the constant multiple and the minimal modularity lifting theorem, and we also discuss the cyclicity of the adjoint Selmer groups.
In particular, we give an affirmative answer to a question of Mazur and Rubin.
Our Stark system construction yields a more refined interpretation of the collection of Flach's zeta elements than the ``geometric Euler system" approach due to Flach, Wiles, Mazur, and Weston.
\end{abstract}

\setcounter{tocdepth}{1}

\section*{Introduction}
\subsection*{Flach's zeta elements}
As explained in the Introduction of his celebrated work \cite{wiles}, A. Wiles first tried to prove the semi-stable Shimura--Taniyama conjecture and Fermat's last theorem via the construction of an Euler system for symmetric squares of semi-stable elliptic curves extending the cohomology classes constructed by M. Flach  \cite{flach-thesis}, which we call \emph{Flach's zeta elements} in this article. However, the construction of the Euler system was incomplete, and it was regarded as a gap.

After that, Wiles, together with R. Taylor, found a completely different approach towards the semi-stable Shimura--Taniyama conjecture, which is now known as the Taylor--Wiles system argument \cite{taylor-wiles}, to fill the gap. 
See \cite{kurihara-wiles-fermat,kolyvagin-fermat,rubin-silverberg-cambridge, darmon-fermat} for the survey articles written before the gap was filled, and see \cite{calegari-survey} for the far-reaching development of the Taylor--Wiles system argument.

This Euler system approach is often known as the ``geometric Euler system" strategy, and it was further developed mainly by B. Mazur and T. Weston \cite{mazur-hecke-curves, mazur-weston, weston-flach, weston-algebraic-cycles}. However, it was still incomplete to obtain the exact bound of Selmer groups since it could proceed only the first step of the Euler system argument. 



The main goals of this article are to investigate the behavior of the collection of Flach's zeta elements beyond the scope of this  geometric Euler system strategy and to deduce certain structural results for adjoint Selmer groups.
However, it is well-known that the construction question of an \emph{Euler} system extending Flach's zeta element is extremely difficult (see \cite[$\S$3.6]{rubin-book}, \cite[pp. 359--360]{weston-flach}, \cite[Rem. 6.3.4]{mazur-rubin-book}, and \cite[Intro.]{loeffler-zerbes-symmetric-square}  for example).

What is the new idea to overcome this notorious difficulty?
We are greatly benefited from various remarkable and significant developments of the theory of Euler systems.
In particular, it turns out that Kolyvagin systems are more essential than Euler systems to obtain the exact bound of the corresponding dual Selmer groups \cite{mazur-rubin-book}. Also, the notion of Stark systems is developed by generalizing the units predicted by Stark-type conjectures \cite{mazur-rubin-control}.
The key observation is that the collection of Flach's zeta elements forms a ``Stark system of Gauss sum type"\footnote{No Gauss sums actually appear in this article.  Since we initially have only \emph{finitely many} cohomology classes forming a Stark system,  we call them ``Gauss sum type".}.
Since the module of Stark systems (equivalently, Kolyvagin systems) is free of rank one, the Gauss sum type Stark system naturally extends to the \emph{bona fide} one.
In other words,  \emph{we construct derivative (so ramified) classes defined over $\mathbb{Q}$ directly from Flach's zeta elements by using the framework of Stark systems}.

Since there is a certain constant multiple subtlety in our Stark system approach,  our result does not give a new proof of the semi-stable Shimura--Taniyama conjecture or Fermat's last theorem yet. 
Indeed, we prove that the minimal modularity lifting theorem is equivalent to showing that the constant multiple is a $p$-adic unit.
Also, we do not know whether our Stark system actually comes from an Euler system\footnote{We are not sure whether it is possible to recover an Euler system from the Kolyvagin system associated to the Euler system.}.

\subsection*{The cyclicity of adjoint Selmer groups}
The cyclicity question has a long history in Iwasawa theory. This is just a quick summary of the preface of Hida's book \cite{hida-elementary-modular} and we strongly recommend the reader to read it.

K. Iwasawa himself was interested in the structure of various classical Iwasawa modules including the cyclicity over the Iwasawa algebra.
In \cite{iwasawa-on-p-adic-l-functions}, he obtained the cyclicity of certain classical Iwasawa modules over the  Iwasawa algebra (as well as the main conjecture) under the Kummer--Vandiver conjecture.
Even without assuming the Kummer--Vandiver conjecture, he proposed several interesting questions on the cyclicity around 1980 \cite{iwasawa-62, iwasawa-U3}.

H. Hida explored the non-abelian analogue of this nature, and \cite{hida-invent-1986} was written with this purpose in mind.
In particular, he studied the cyclicity question of the adjoint Selmer groups over the universal deformation ring in a series of his papers\cite{hida-anticyclotomic-cyclicity, hida-cyclicity-fundamental-iwasawa-2017,hida-universal-ordinary-real-quadratic, hida-elementary-modular}.
In \cite[Chap. 7]{hida-elementary-modular}, he proved many cyclicity results for the adjoint Selmer groups of Artin representations.
As far as we understand, the elliptic curve case seems quite open \cite[pp. 287--288 and 339]{hida-elementary-modular} unless the length of adjoint Selmer groups is $\leq 1$. As an application of the above Stark system construction, we obtain a cyclicity result for the adjoint Selmer groups of  elliptic curves (Corollary \ref{cor:cyclicity}).

\subsection*{Comparison with other recent work}
In recent years, some groups of mathematicians have developed new approaches towards the construction of symmetric square and adjoint Euler systems from different sources \cite{loeffler-zerbes-symmetric-square, urban-adjoint, loeffler-zerbes-adjoint, sangiovanni-skinner}.
We provide brief remarks on their constructions and explain why they are all independent of ours.

In \cite{loeffler-zerbes-symmetric-square}, 
the Beilinson--Flach Euler system for the Rankin--Selberg product of the same $p$-ordinary modular form is studied
and is constructed from Beilinson--Flach's elements and bounds Selmer groups of the symmetric square representations of an ordinary modular form twisted by a non-trivial Dirichlet character. This twist is inevitable to have the correct Euler system relation.
In \cite{urban-adjoint}, an Euler system for adjoint representations of ordinary Hilbert modular forms is constructed directly from congruence modules, but the $R=\mathbb{T}$ theorem  is an ingredient of the construction.
Also, the corresponding explicit reciprocity law depends on a conjecture on the Fitting ideals of some equivariant congruence modules for abelian base change.
In \cite{loeffler-zerbes-adjoint}, an Euler system for adjoint representations of ordinary modular forms is constructed from Asai--Flach classes and is used to prove one-sided divisibility of the main conjecture for the symmetric square Selmer group over the cyclotomic $\mathbb{Z}_p$-extension of $\mathbb{Q}$ up to a power of $p$.
In \cite{sangiovanni-skinner}, an Euler system for the symmetric square representation of an ordinary modular form twisted by an odd Dirichlet character is initially constructed by utilizing the pull-back of Eisenstein series on $\mathrm{SO}(3,2)$. Then by using Hida theory (among others), an Euler system for the symmetric square representation (with no extra twist) is obtained. As its consequence, the $R=\mathbb{T}$ theorem follows.
The $p$-ordinary condition is assumed in all the work mentioned above.

Our Stark system also detects the precise upper bound of the adjoint Selmer groups, so the direct connection with the modularity lifting theorem is obtained.
More precisely, our construction yields an equivalence statement as mentioned earlier.
It appears that our approach is the most akin to Wiles' original method since we work with Flach's zeta elements as Wiles did, so it might have some historical meaning.
Also, our argument works for any good reduction prime. 
In addition, the structural refinement like the cyclicity is not observed in other approaches.

\subsection*{The organization}
In \S\ref{sec:main-results}, we state our main result on the existence of the rank one primitive Stark systems for symmetric square representations and its arithmetic applications.
In \S\ref{sec:flach-zeta-elements}, we review Flach's zeta elements and state the explicit reciprocity law for Flach's zeta elements.
In \S\ref{sec:verifying-hypotheses}, we verify the working hypotheses for rank one Stark systems.
In \S\ref{sec:explicit-reciprocity-law}, we recall Weston's computation on the explicit reciprocity law.
In \S\ref{sec:construction},  we prove our main result by constructing the Stark system explicitly.  This is the heart of this article.
In \S\ref{sec:corollaries}, we prove corollaries to the main result.
\section{Statement of the main result} \label{sec:main-results}
Let $E$ be a semi-stable \emph{modular}\footnote{We do not use the semi-stable modularity theorem of Wiles and Taylor--Wiles \cite{wiles,taylor-wiles} and its generalizations in this article.} elliptic curve over $\mathbb{Q}$ of conductor $N$ and $p \geq 5$ a good reduction prime for $E$. Denote by $T_pE$ the $p$-adic Tate module for $E$. Assume that  $\overline{\rho} : G_{\mathbb{Q}} = \mathrm{Gal} ( \overline{\mathbb{Q}}/\mathbb{Q} ) \to \mathrm{Aut}_{\mathbb{F}_p}(E[p])$ is irreducible and $N = N(\overline{\rho})$ where $N(\overline{\rho})$ is the conductor of $\overline{\rho}$.
For $k \geq 1$, let $\mathrm{ad}^0(E[p^k])$ is the (trace zero) adjoint representation associated to $E[p^{k}]$.

We briefly describe how a Stark system (Definition \ref{defn:stark-systems}) looks like and all the details can be found in \S\ref{sec:construction}.

Let $\ell$ be a Flach prime (Definition \ref{defn:flach-prime}), i.e. a prime such that $(Np, \ell) = 1$, $\ell \equiv 1 \pmod{p}$, and $a_\ell(E) \equiv \alpha + \alpha^{-1} \not\equiv \pm 2 \pmod{p}$ for some $\alpha$ with $\alpha^2 \not\equiv 1 \pmod{p}$.

A Stark system $\stark^{\ell}$ for the Selmer data 
$(\mathrm{Sym}^2(E[p^k]), \mathcal{F}^{\ell}_{\mathrm{BK}}, \mathcal{P}^{\mathrm{Flach}, (\ell)}_k)$ (reviewed in Definition \ref{defn:selmer-data}) is a compatible family of cohomology class 
$$\epsilon^{\ell}_n \in  \bigwedge^{1+\nu(n)} \mathrm{Sel}_{\mathcal{F}^{\ell n}_{\mathrm{BK}}} ( \mathbb{Q}, \mathrm{Sym}^2(E[p^k]) )\otimes \bigwedge^{\nu(n)} \bigoplus_{q \vert n} \mathrm{Hom} \left( \frac{\mathrm{H}^1( \mathbb{Q}_{q}, \mathrm{Sym}^2(E[p^k]) )}{\mathrm{H}^1_{f}( \mathbb{Q}_{q}, \mathrm{Sym}^2(E[p^k]) )}, \mathbb{Z}/p^k\mathbb{Z} \right)$$
where $n$ is a square-free product of the primes in $\mathcal{P}^{\mathrm{Flach}, (\ell)}_k$, and $\nu(n)$ is the number of prime divisors of $n$, $\mathcal{F}^{\ell n}_{\mathrm{BK}}$ is the $\ell n$-relaxed Bloch--Kato Selmer structure, and $\mathrm{H}^1_{f}$ means the finite part of the local cohomology group.
A Stark system for $\mathrm{Sym}^2(T_pE)$ is defined by taking the inverse limit.

The main goal of this article is to construct the (rank one) Stark system from Flach's zeta elements. It partially resolves a question of Mazur--Rubin in \cite[Rem. 6.3.4]{mazur-rubin-book}. Indeed, we completely settle their question under the semi-stable modularity lifting theorem (see Corollary \ref{cor:bloch-kato} below).
\begin{thm}[Flach--Stark systems] \label{thm:flach-stark}
For a Flach prime $\ell$, there exists the non-trivial primitive rank one Stark system 
$$\stark^{\mathrm{Flach}', \ell}$$
 for $\mathrm{Sym}^2(T_pE)$ extending a constant multiple of Flach's zeta element $c(\ell)$ (recalled in Theorem \ref{thm:flach-zeta-elements}) which determines the structure of the $\ell$-strict adjoint Bloch--Kato Selmer group
 $$\mathrm{Sel}_{\ell\textrm{-}\mathrm{str}}(\mathbb{Q}, \mathrm{ad}^0(E[p^\infty])).$$
In particular,  the constant multiple  is independent of $\ell$.
\end{thm}
 We call $\stark^{\mathrm{Flach}', \ell}$ the primitive normalization of the \textbf{Flach--Stark system} $\stark^{\mathrm{Flach}, \ell}$. 
\begin{proof}
We explicitly construct $\stark^{\mathrm{Flach}, \ell}$ from Flach's zeta elements and define 
$\stark^{\mathrm{Flach}', \ell}$ by multiplying the constant multiple in \S\ref{sec:construction}.
\end{proof}
In the rest of this section, we fix a prime $\ell$ satisfying the conditions in Theorem \ref{thm:flach-stark}.
Let $\ks^{\mathrm{Flach}', \ell}$ be the the primitive normalization of the (non-trivial) \textbf{Flach--Kolyvagin system} $\ks^{\mathrm{Flach}, \ell}$ corresponding to $\stark^{\mathrm{Flach}', \ell}$ under the isomorphism between the modules of Stark systems and Kolyvagin systems under our setting \cite[Rem. 11.5 and Thm. 12.4]{mazur-rubin-control}. 
Indeed, the constant multiple in Theorem \ref{thm:flach-stark} is designed to deduce the following formula for the exact bound of the genuine Bloch--Kato Selmer group in terms of $\ks^{\mathrm{Flach}', \ell}$.
Let 
$$\mathrm{loc}^s_\ell : \mathrm{H}^1(\mathbb{Q}, \mathrm{Sym}^2(T_pE) ) \to
\mathrm{H}^1(\mathbb{Q}_\ell, \mathrm{Sym}^2(T_pE) ) \to
\mathrm{H}^1_{/f}(\mathbb{Q}_\ell, \mathrm{Sym}^2(T_pE) )  $$
be the singular quotient of the localization map at $\ell$ where the first map is the localization at $\ell$ and the second map is the natural quotient, and
$\mathrm{H}^1_{/f}(\mathbb{Q}_\ell, \mathrm{Sym}^2(T_pE) )  = \frac{\mathrm{H}^1(\mathbb{Q}_\ell, \mathrm{Sym}^2(T_pE) ) }{\mathrm{H}^1_f(\mathbb{Q}_\ell, \mathrm{Sym}^2(T_pE) ) }$.
\begin{cor}[Exact bound] \label{cor:exact-selmer-size}
Let 
$$\mathrm{loc}^s_\ell ( \ks^{\mathrm{Flach}', \ell} ) = \left\lbrace \mathrm{loc}^s_\ell ( \kappa^{\mathrm{Flach}, \ell}_n ) : n \in \mathcal{N}^{\mathrm{Flach},(\ell)}_1 \right\rbrace$$
be the singular quotient of the localization of $\ks^{\mathrm{Flach}, \ell}$ at $\ell$ where $\mathcal{N}^{\mathrm{Flach},(\ell)}_1$ is defined in Definition \ref{defn:flach-prime}.
Then
$$\mathrm{ord}_p ( \mathrm{loc}^s_\ell ( \kappa^{\mathrm{Flach}', \ell}_1 ) ) = \mathrm{length}_{\mathbb{Z}_p} ( \mathrm{Sel}(\mathbb{Q}, \mathrm{ad}^0(E[p^\infty])) ) $$
where $\mathrm{Sel}(\mathbb{Q}, \mathrm{ad}^0(E[p^\infty]))$ is the adjoint Bloch--Kato Selmer group and
$\mathrm{ord}_{p} ( \mathrm{loc}^s_\ell(\kappa^{\mathrm{Flach}', \ell}_1) )$ is the $p$-divisibility index of $\mathrm{loc}^s_\ell(\kappa^{\mathrm{Flach}', \ell}_1)$ in $\mathrm{H}^1_{/f}(\mathbb{Q}_\ell, \mathrm{Sym}^2(T_pE) )$.
 \end{cor}
\begin{proof}
See $\S$\ref{subsec:cor-exact-selmer-size}.
\end{proof}
\begin{rem}
It is known that the primitivity of a (rank one) Kolyvagin system is closely related to showing the corresponding main conjecture \cite{mazur-rubin-book, kks,  sakamoto-p-selmer,kim-structure-selmer}. This principle would apply to our setting as follows if such a formulation is valid.
Let $\mathbb{Q}_\infty$ be the cyclotomic $\mathbb{Z}_p$-extension of $\mathbb{Q}$ and $\Lambda$ the cyclotomic Iwasawa algebra.
\emph{If} the $\Lambda$-adic deformation $\ks^{\mathrm{Flach}', \ell, \infty}$ of $\ks^{\mathrm{Flach}', \ell}$ exists
and $\kappa^{\mathrm{Flach}, \ell, \infty}_1$ is non-zero in $\widehat{\mathrm{Sel}_{\ell\textrm{-}\mathrm{rel}}}(\mathbb{Q}_\infty, \mathrm{Sym}^2(T_pE))$, then the \emph{hypothetical} main conjecture
$$``\mathrm{char}_{\Lambda} \left( \widehat{\mathrm{Sel}_{\ell\textrm{-}\mathrm{rel}}}(\mathbb{Q}_\infty, \mathrm{Sym}^2(T_pE)) / \Lambda \kappa^{\mathrm{Flach}', \ell, \infty}_1 \right)
= \mathrm{char}_{\Lambda} \left( \mathrm{Sel}_{\ell\textrm{-}\mathrm{str}}(\mathbb{Q}_\infty, \mathrm{ad}^0(E[p^\infty]))^\vee \right)"$$
follows from the primitivity of $\ks^{\mathrm{Flach}', \ell}$ \cite[$\S$5.3]{mazur-rubin-book}
where 
$\widehat{\mathrm{Sel}_{\ell\textrm{-}\mathrm{rel}}}(\mathbb{Q}_\infty, -) = \varprojlim_n\mathrm{Sel}_{\ell\textrm{-}\mathrm{rel}}(\mathbb{Q}_n, -)$
 with respect to the corestriction maps and
 $(-)^\vee$ is the Pontryagin dual.
However, the construction of the $\Lambda$-adic deformation $\ks^{\mathrm{Flach}', \ell, \infty}$ of $\ks^{\mathrm{Flach}', \ell}$ is completely out of reach yet.
Furthermore, since the local condition of the adjoint and symmetric square Selmer groups at $p$ is the Bloch--Kato one, a certain $p$-ordinary assumption would be needed to have the $\Lambda$-cotorsion property for $\mathrm{Sel}_{\ell\textrm{-}\mathrm{str}}(\mathbb{Q}_\infty, \mathrm{ad}^0(E[p^\infty]))$.
\end{rem}


Denote by $\mathrm{deg}(\phi)$ the minimal modular degree of $E$ where $\phi : X_0(N) \to E$ is the modular parametrization.
We obtain the following connection between Flach's zeta elements and the minimal modularity lifting theorem from Theorem \ref{thm:flach-stark}.
\begin{cor}[Minimal modularity lifting] \label{cor:bloch-kato}
The following three statements are equivalent:
\begin{enumerate}
\item The constant multiple in Theorem \ref{thm:flach-stark} is a $p$-adic unit so that
$$\mathrm{ord}_p  ( \mathrm{loc}^s_\ell c(\ell) ) =  \mathrm{length}_{\mathbb{Z}_p} \mathrm{Sel}(\mathbb{Q}, \mathrm{ad}^0(E[p^\infty])) .$$
\item The Bloch--Kato conjecture  for the adjoint representation of $E$ holds, i.e.
$$\mathrm{ord}_p( \mathrm{deg}(\phi) ) = \mathrm{length}_{\mathbb{Z}_p} \mathrm{Sel}(\mathbb{Q}, \mathrm{ad}^0(E[p^\infty]))  .$$
\item The minimal deformation ring of $\overline{\rho}$ and the minimal Hecke algebra localized at the maximal ideal $\mathfrak{m}_{\overline{\rho}}$ associated to $\overline{\rho}$ are isomorphic as complete intersections, i.e.
$$R^{\mathrm{min}}_{\overline{\rho}} \simeq \mathbb{T}^{\mathrm{min}}_{\mathfrak{m}_{\overline{\rho}}} .$$
\end{enumerate}
\end{cor}
\begin{proof}
See $\S$\ref{subsec:cor-bloch-kato}.
\end{proof}
Of course, all the statements in Corollary \ref{cor:bloch-kato} are true thanks to the work of Wiles and Taylor--Wiles \cite{wiles, taylor-wiles}.
One remarkable feature of Corollary \ref{cor:bloch-kato} is that the modularity lifting theorem can be deduced from a certain property regarding Flach's zeta elements (finally!). 
More precisely, the final step in deducing the modularity lifting theorem from our Stark system argument is to control the constant multiple.
 
On the other hand,  the modularity lifting theorem (Corollary \ref{cor:bloch-kato}.(3)) yields the explicit construction of the rank one primitive Kolyvagin system for $ \mathrm{Sym}^2(T_pE)$ extending Flach's zeta elements (without introducing any constant multiple) for every elliptic curve $E$ which occurs in the minimal deformation space of $\overline{\rho}$\footnote{We thank H. Hida for bringing this deformation-theoretic observation to our attention. }. 

In this sense, Corollary \ref{cor:bloch-kato} illustrates both the possibility and the limitation of the approach towards modularity lifting theorem based on Flach's zeta elements.

Although we do not know whether or how $\mathrm{loc}^s_\ell ( \ks^{\mathrm{Flach}', \ell} )$ determines the structure of $\mathrm{Sel}(\mathbb{Q}, \mathrm{ad}^0(E[p^\infty]))$ in general yet\footnote{In \cite{kim-structure-selmer}, the self-duality of Galois representations is used in an essential way.  We are informed that R. Sakamoto is working on the general structure theorem with rank \emph{zero} Stark/Kolyvagin systems,  and it seems to be closely related to our approach.}, we still have the following structural application under the $p$-indivisibility assumption on $c(\ell)$. 
\begin{cor}[Cyclicity] \label{cor:cyclicity}
Suppose that any (equivalent) statement of Corollary \ref{cor:bloch-kato} holds.
If we further assume that $c(\ell) \in \mathrm{Sel}_{\ell\textrm{-}\mathrm{rel}}(\mathbb{Q} , \mathrm{Sym}^2(T_pE))$ is not divisible by $p$ for one $\ell$, then the adjoint Selmer group is cyclic, i.e.
$$\mathrm{Sel}(\mathbb{Q}, \mathrm{ad}^0(E[p^\infty])) \simeq \mathbb{Z}_p / \mathrm{deg}(\phi) \mathbb{Z}_p $$
where $\mathrm{Sel}_{\ell\textrm{-}\mathrm{rel}}(\mathbb{Q} , \mathrm{Sym}^2(T_pE))$ is the $\ell$-relaxed compact symmetric square Bloch--Kato Selmer group.
\end{cor}
\begin{proof}
See $\S$\ref{subsec:cor-cyclicity}.
\end{proof}
%

This cyclicity can be understood as the structural refinement of the Bloch--Kato conjecture (cf. \cite{kim-gross-zagier, kim-structure-selmer}).
We make the following question (not conjecture).
\begin{ques} \label{ques:cyclicity}
Does the $p$-indivisibility assumption in Corollary \ref{cor:cyclicity} hold in general?
It is equivalent to proving that $\mathrm{Sel}_{\ell\textrm{-}\mathrm{str}}(\mathbb{Q}, \mathrm{ad}^0(E[p^\infty]))$ is trivial.
\end{ques}
We are informed from Hida that the adjoint Selmer group should be cyclic when the level is minimal, i.e.  $N = N(\overline{\rho})$.
Under the cyclicity assumption, Question \ref{ques:cyclicity} becomes a dichotomy:
\begin{enumerate}
\item 
$\mathrm{Sel}( \mathbb{Q}, \mathrm{ad}^0(E[p^\infty])) \simeq
\left(
\frac{ \mathrm{H}^1_{/f}(\mathbb{Q}_\ell , \mathrm{Sym}^2(T_pE) )  }{ \mathrm{loc}^s_\ell \left( \mathrm{Sel}_{\ell\textrm{-}\mathrm{rel}}( \mathbb{Q}, \mathrm{Sym}^2(T_pE))  \right) } \right)^\vee
\simeq \mathbb{Z}_p / \mathrm{deg}(\phi) \mathbb{Z}_p
$
where $(-)^\vee$ is the Pontryagin dual. This is equivalent to the $p$-indivisibility of $c(\ell)$.
\item 
$\mathrm{Sel}( \mathbb{Q}, \mathrm{ad}^0(E[p^\infty])) \simeq \mathrm{Sel}_{\ell\textrm{-}\mathrm{str}}( \mathbb{Q}, \mathrm{ad}^0(E[p^\infty])) \simeq \mathbb{Z}_p / \mathrm{deg}(\phi) \mathbb{Z}_p$
and 
$ \mathrm{loc}^s_\ell : \mathrm{Sel}_{\ell\textrm{-}\mathrm{rel}}( \mathbb{Q}, \mathrm{Sym}^2(T_pE))    
\to \mathrm{H}^1_{/f}(\mathbb{Q}_\ell , \mathrm{Sym}^2(T_pE) )$ is surjective.
This is equivalent to the $p$-divisibility of $c(\ell)$.
\end{enumerate}
The higher weight generalization of Theorem \ref{thm:flach-stark} and its interpretation in terms of rank \emph{zero} Stark/Kolyvagin systems is being investigated in a joint work in progress with Ryotaro Sakamoto.
\section{Flach's zeta elements and the explicit reciprocity law} \label{sec:flach-zeta-elements}
\subsection{}
Let $G_k \subseteq \mathrm{GL}_2(\mathbb{Z}/p^k\mathbb{Z})$ be the image of $G_{\mathbb{Q}}$ under the mod $p^k$ reduction of $\rho : G_{\mathbb{Q}} \to \mathrm{Aut}_{\mathbb{Z}_p} (T_pE) \simeq \mathrm{GL}_2(\mathbb{Z}_p)$.
The following annihilation result is the very starting point of this story.
\begin{thm}[Flach] \label{thm:flach-annihilation}
Assume that
$p$ is odd,
$\overline{\rho}$ is absolutely irreducible,
$p$ does not divide $N$, and
 $N = N(\overline{\rho})$.
If we further assume that $N$ is square-free and $\mathrm{H}^1(G_k, \mathrm{ad}^0(E[p^k])) = 0$ for all $k \geq 1$, then
$$\mathrm{deg}(\phi) \cdot \mathrm{Sel}(\mathbb{Q}, \mathrm{ad}^0(E[p^\infty])) = 0.$$
In particular, $\mathrm{Sel}(\mathbb{Q}, \mathrm{ad}^0(E[p^\infty])) $ is finite, and it is trivial for all but finitely many primes $p$.
\end{thm}
\begin{proof}
See \cite[Thm. 1]{flach-annihilation} and \cite[Thm. 1]{flach-thesis}.
\end{proof}
In order to obtain Theorem \ref{thm:flach-annihilation}, Flach constructed the following important cohomology classes, which we call \textbf{Flach's zeta elements}.
\begin{thm}[Flach's zeta elements] \label{thm:flach-zeta-elements}
Assume that
$p$ is odd,
$\overline{\rho}$ is absolutely irreducible,
$p$ does not divide $N$, and
 $N = N(\overline{\rho})$.
For a prime $\ell$ not dividing $2Np$, 
there exists a cohomology class 
$$c(\ell) \in \mathrm{Sel}_{\ell\textrm{-}\mathrm{rel}}(\mathbb{Q}, \mathrm{Sym}^2(T_pE))$$
such that
$$\mathrm{ord}_{p} ( \mathrm{loc}^s_\ell(c(\ell)) ) \leq \mathrm{ord}_{p} (\mathrm{deg}(\phi))  + \mathrm{ord}_{p} ( \alpha_\ell - \beta_\ell ) $$
where
 $\alpha_\ell$ and $ \beta_\ell$ are the roots of $X^2 - a_\ell(E)X +\ell$.
\end{thm}
\begin{proof}
See \cite[Prop. 1 and $\S$5]{flach-annihilation}.
\end{proof}
\subsection{}
We explain how to choose auxiliary primes for the Chebotarev density argument. The following lemma is useful.
\begin{lem} \label{lem:surjectivity}
Let $E$ be a semi-stable elliptic curve over $\mathbb{Q}$, $p$ a prime, and $\overline{\rho} : G_{\mathbb{Q}} \to \mathrm{Aut}_{\mathbb{F}_p}(E[p])$.
Then $\overline{\rho}$ is surjective if and only if $\overline{\rho}$ is irreducible.
\end{lem}
\begin{proof}
See \cite[Prop. 2.1]{edixhoven-serre-conjecture}.
\end{proof}
When $p > 3$, the surjectivity of $\overline{\rho}$ implies that the image of $\rho$ contains $\mathrm{SL}_2(\mathbb{Z}_p)$ \cite{serre-abelian-l-adic}.
\begin{choice} \label{choice:tau-2}
Suppose that the image of $\rho$ contains $\mathrm{SL}_2(\mathbb{Z}_p)$.
Then we (are able to) choose a Galois element
$\tau \in G_{\mathbb{Q}(\zeta_{p^\infty})}$ such that 
$$\rho(\tau) \sim \begin{pmatrix}
\alpha & 0 \\
0 & \alpha^{-1}
\end{pmatrix}$$
where 
\begin{itemize}
\item $\alpha \in \mathbb{Z}_{p}$ with $\alpha^2 \not\equiv 1 \pmod{p}$ when $p >3$, or
\item $\alpha \in \mathbb{Z}_{p^2}$ with $\alpha \equiv \sqrt{-1} \in \mathbb{F}_{3^2}$ when $p =3$.
\end{itemize}
\end{choice}
This choice of $\tau$ determines the type of auxiliary primes when we apply the Chebotarev density theorem.
Let $\ell$ be a good reduction prime for $E$ with $\ell \equiv 1 \pmod{p^k}$
and denote by $\mathrm{Fr}_{\ell}$ the arithmetic Frobenius at $\ell$.
If $\mathrm{Fr}_{\ell}$ acts on $T_pE$ with eigenvalue  $\alpha$ and $\alpha^{-1} \cdot \ell$,
then we have $a_\ell(E) = \alpha + \alpha^{-1} \cdot \ell$.
\begin{defn} \label{defn:flach-prime}
Let $k \geq 1$ be an integer.
A prime $\ell$ is said to be a \textbf{$k$-Flach prime} if $(\ell, Np)=1$, $\ell \equiv 1 \pmod{p^k}$, and 
$\mathrm{Fr}_{\ell}$ acts on $E[p^k]$ as $\rho(\tau) \pmod{p^k}$ 
where $\tau$ is chosen in Choice \ref{choice:tau-2}.
This implies that
$a_\ell(E) = \alpha + \alpha^{-1} \cdot \ell \equiv \alpha + \alpha^{-1}  \not\equiv \pm 2 \pmod{p^k}$
with $\alpha$ in Choice \ref{choice:tau-2}; in particular, $\alpha^2 \not\equiv 1 \pmod{p}$.
We call 1-Flach primes by Flach primes for convenience. 
Denote by $\mathcal{P}^{\mathrm{Flach},(\ell)}_k$ the set of all $k$-Flach primes except $\ell$ for fixed $(E,p)$
and by $\mathcal{N}^{\mathrm{Flach},(\ell)}_k$ the set of square-free products of primes in $\mathcal{P}^{\mathrm{Flach},(\ell)}_k$.
\end{defn}
\begin{rem}
The notion of $k$-Flach primes depends on Choice \ref{choice:tau-2}.
In the case of Heegner points, $\tau$ is chosen to be the complex conjugation.
In \cite{flach-thesis}, $\tau$ is chosen to be the complex conjugation again.
For Kato's Euler systems, $\tau$ is chosen to be a unipotent element as explained in \cite[Prop. 3.5.8]{rubin-book}. 
\end{rem}
\subsection{}
It is easy to check that $\alpha_\ell - \beta_\ell$ is a $p$-adic unit for a Flach prime $\ell$.
From Theorem \ref{thm:flach-zeta-elements}, we have inequality
$$\mathrm{ord}_p(\mathrm{loc}^s_\ell(c(\ell)) ) \leq  \mathrm{ord}_p (\mathrm{deg}(\phi)) $$
for a Flach prime $\ell$. 
Here, $\mathrm{deg}(\phi)$ is also called the \emph{depth of a partial geometric Euler system} in the sense of Mazur and Weston \cite{mazur-hecke-curves,weston-flach, weston-algebraic-cycles}.

This inequality can be upgraded to equality thanks to Weston's computation \cite{weston-flach, weston-algebraic-cycles}.
The equality (\ref{eqn:explicit-reciprocity-law}) below should be viewed as the \emph{explicit reciprocity law for Flach's zeta elements} since  $\mathrm{deg}(\phi)$ has a precise connection with the adjoint $L$-value via the formula of Shimura and Hida.
\begin{prop}[Weston's explicit reciprocity law] \label{prop:explicit-reciprocity-law}
If $\ell$ is a Flach prime,  the inequality in Theorem \ref{thm:flach-zeta-elements} becomes equality
\begin{equation} \label{eqn:explicit-reciprocity-law}
\mathrm{ord}_p ( \mathrm{loc}^s_\ell c(\ell) ) = \mathrm{ord}_p ( \mathrm{deg}(\phi) )  .
\end{equation}
\end{prop}
\begin{proof}
See $\S$\ref{sec:explicit-reciprocity-law}.
\end{proof}
The collection of $c(\ell)$ varying Flach primes $\ell$ forms a (cohesive) Flach system of depth $\mathrm{deg}(\phi)$ in the sense of Mazur and Weston \cite{mazur-hecke-curves, weston-algebraic-cycles}.

\section{Verifying the hypotheses for Stark systems} \label{sec:verifying-hypotheses}
We review the running hypotheses of Stark systems (of core rank one) in \cite[$\S$4]{mazur-rubin-control} and verify them for our setting.
\subsection{Running hypotheses}
Let $\Sigma = \lbrace p, \infty, \textrm{ramified primes for } T_p E \rbrace$ and fix a Flach prime $\ell$.
\begin{defn} \label{defn:selmer-data}
By \textbf{Selmer data}, we mean the triple 
$$(\mathrm{Sym}^2(T_p E), \mathcal{F}^{\ell}_{\mathrm{BK}}, \mathcal{P}^{(\ell)})$$ where
\begin{itemize}
\item $\mathrm{Sym}^2(T_p E)$ as a $G_{\mathbb{Q}}$-module,
\item $\mathcal{F}^{\ell}_{\mathrm{BK}} $ is the Bloch--Kato Selmer structure except the relaxed local condition at $\ell$, and
\item $\mathcal{P}^{(\ell)}$ is the set of primes disjoint from $\Sigma$ and $\lbrace \ell \rbrace$. In practice, we will utilize $\mathcal{P}^{\mathrm{Flach}, (\ell)}_k$, the set of $k$-Flach primes excluding $\ell$ for some integer $k \geq 1$.
\end{itemize}
\end{defn}
For a $G_{\mathbb{Q}}$-module $A$, we write $\mathbb{Q}(A)$ for the fixed field in $\overline{\mathbb{Q}}$ of the kernel of the map 
$G_{\mathbb{Q}} \to \mathrm{Aut}(A)$.
\begin{assu}
We list the assumptions to proceed the Stark system argument \cite[$\S$4]{mazur-rubin-control}:
\begin{enumerate}
\item[(H.1)] $\mathrm{H}^0(\mathbb{Q}, \mathrm{Sym}^2(E[p])) = \mathrm{H}^0(\mathbb{Q}, \mathrm{ad}^0(E[p])) = 0$ and $\mathrm{Sym}^2(E[p])$ is absolutely irreducible.
\item[(H.2)] There exists an element $\tau \in G_{\mathbb{Q}(\zeta_{p^\infty})} = \mathrm{Gal}(\overline{\mathbb{Q}} / \mathbb{Q}(\zeta_{p^\infty}) )$ such that 
$\mathrm{Sym}^2(T_p E) / (\tau - 1) \mathrm{Sym}^2(T_p E)$ is free of rank one over $\mathbb{Z}_p$.
\item[(H.3)] $\mathrm{H}^1(\mathbb{Q}_{T}/\mathbb{Q}, \mathrm{Sym}^2(E[p])) = \mathrm{H}^1(\mathbb{Q}_{T}/\mathbb{Q}, \mathrm{ad}^0(E[p])) = 0$
where $\mathbb{Q}_{T} = \mathbb{Q}(\mathrm{Sym}^2(T_pE), \zeta_{p^\infty})$.
\item[(H.4)] $\mathrm{Sym}^2(E[p]) \not\simeq \mathrm{ad}^0(E[p])$ as Galois modules or $p >3$.
\item[(H.5)] The Selmer structure $\mathcal{F}^{\ell}_{\mathrm{BK}}$ is cartesian.
\item[(H.6)] $\chi(\mathrm{Sym}^2(T_p E), \mathcal{F}^{\ell}_{\mathrm{BK}}) = 1$, where $\chi(\mathrm{Sym}^2(T_p E), \mathcal{F}^{\ell}_{\mathrm{BK}})$ is the core rank of $(\mathrm{Sym}^2(T_p E), \mathcal{F}^{\ell}_{\mathrm{BK}})$.
\item[(H.7)] $I_q = 0$ for $q \in \mathcal{P}^{\mathrm{Flach}, (\ell)}_k$ when the coefficient ring is artinian (e.g. $\mathbb{Z}/p^k\mathbb{Z}$).
\end{enumerate}
\end{assu}
\subsection{Verifying (H.1)}
(H.1) follows from Lemma \ref{lem:surjectivity} and \cite[$\S$3.6]{rubin-book}.
\subsection{Verifying (H.2)}
Thanks to Lemma \ref{lem:surjectivity}, we use $\tau$ in Choice \ref{choice:tau-2}.
Then the induced action of $\tau$ on $\mathrm{Sym}^2(T_pE)$ has eigenvalues $\alpha^{2}$, 1, and $\alpha^{-2}$.
Since $\alpha^2 \not\equiv 1 \pmod{p}$, 
$\mathrm{Sym}^2(T_p E) / (\tau - 1) \mathrm{Sym}^2(T_p E)$ is free of rank one over $\mathbb{Z}_p$.
See \cite[$\S$3.6]{rubin-book}.
\subsection{Verifying (H.3)}
(H.3) follows from \cite[$\S$3.6]{rubin-book}.
See also \cite[Lem. 1.2]{flach-thesis} with \cite[Rem. 1 in $\S$4]{flach-annihilation}, and \cite[Lem. 2.48]{ddt} when $p > 5$.
\subsection{Verifying (H.4)}
This is immediate.
\subsection{Verifying (H.5)}
(H.5) follows from \cite[Lem. 3.7.1 and Rem. 3.7.2]{mazur-rubin-book}.

\subsection{Verifying (H.6)} \label{subsubsec:verifying_h_6}
We compute the core ranks of certain Selmer structures we concern.
In particular, we confirm all the statements on the core ranks in \cite[Rem. 6.3.4]{mazur-rubin-book}.
\subsubsection{Canonical Selmer structure}
We first compute the core rank for the ``canonical" Selmer structure for $\mathrm{Sym}^2(T_pE)$.
Indeed, the notions of the canonical Selmer structures in \cite[Def. 3.2.1]{mazur-rubin-book} and \cite[Def. 5.1]{mazur-rubin-control} are slightly different at $p$.
It is the relaxed condition at $p$ in the former, but the universal norm subgroup is involved in the latter.
By \cite[Cor. 5.3]{mazur-rubin-control}, they become equivalent if $\mathrm{H}^0(\mathbb{Q}_p, \mathrm{ad}^0(E[p^\infty]))$ has finite length.
The following lemma confirms that our case satisfies this condition.
\begin{lem} \label{lem:local-corank-zero}
For a non-CM elliptic curve $E$ and a good reduction prime $p > 2$, we have 
$$\mathrm{cork}_{\mathbb{Z}_p}\mathrm{H}^0(\mathbb{Q}_p, \mathrm{ad}^0(E[p^\infty])) = 0.$$
\end{lem}
Since we are working on semi-stable elliptic curves, the non-CM hypothesis is automatic.
\begin{proof}
Let $V_pE = T_pE \otimes \mathbb{Q}_p$.
Since
$$\mathrm{cork}_{\mathbb{Z}_p}\mathrm{H}^0(\mathbb{Q}_p, \mathrm{ad}^0(E[p^\infty])) = \mathrm{dim}_{\mathbb{Q}_p}\mathrm{H}^0(\mathbb{Q}_p, \mathrm{ad}^0(V_pE)),$$
it suffices to show that
$\mathrm{End}_{ G_{\mathbb{Q}_p}} (V_pE) \simeq \mathbb{Q}_p$.
For convenience, we write $V$ for $V_pE$ as a representation of $G_{\mathbb{Q}_p}$.

When $E$ has good ordinary reduction at $p$, the non-CM hypothesis implies that $V$ is isomorphic to a non-split extension of two distinct characters as a local Galois representation at $p$:
\[
\xymatrix{
0 \ar[r] & F^+V \ar[r] & V \ar[r] & F^-V \ar[r] & 0 .
}
\]
Since $F^+V$ and $F^-V$ are distinct, there is no non-zero $G_{\mathbb{Q}_p}$-equivariant homomorphism between $F^+V$ and $F^-V$, so every $\psi  \in \mathrm{End}_{ G_{\mathbb{Q}_p}} (V)$ preserves $F^+V$.
Thus, the restriction of $\psi$ to $F^+V$ and the induced map on $F^-V$ are both scalar multiplication maps.
Suppose that both scalars are different, say $c_+$ and $c_-$ with $c_+ \neq c_-$.
For this $\psi$, define another $G_{\mathbb{Q}_p}$-equivariant endomorphism
$$\psi' =  \dfrac{\psi - c_- \cdot \mathrm{id}}{c_+ - c_-} .$$
Since $\psi'$ becomes the identity map on $F^+V$ and the zero map on $F^-V$,
$V$ splits into the direct sum of $F^+V$ and $\mathrm{ker}(\psi')$ as a $G_{\mathbb{Q}_p}$-representation.
Since this violates the non-CM hypothesis, we have $c_+ = c_-$, just say $c$.
Then the restriction of $\psi - c \cdot \mathrm{id}$ to $F^+V$ is the zero map and
the projection of the image of $\psi - c \cdot \mathrm{id}$ to $F^-V$ also becomes the zero map.
In other words, the image of $\psi - c \cdot \mathrm{id}$ lies in $F^+V$.
Since $\mathrm{Hom}_{ G_{\mathbb{Q}_p}}(F^-V, F^+V) = 0$, we have 
$$\psi = c \cdot \mathrm{id}$$
on $V$. This implies that $\mathrm{End}_{ G_{\mathbb{Q}_p}} (V) \simeq \mathbb{Q}_p$.

We now suppose that $E$ has supersingular reduction at $p$.
Since $V$ is irreducible, every non-zero $\psi \in \mathrm{End}_{ G_{\mathbb{Q}_p}} (V)$ has trivial kernel; thus, it is an automorphism. In other words, $\mathrm{End}_{ G_{\mathbb{Q}_p}} (V)$ is a division algebra.
Consider Fontaine's de Rham functor
$\mathbf{D}_{\mathrm{dR}}(V) = \mathrm{H}^0(\mathbb{Q}_p, V \otimes \mathbf{B}_{\mathrm{dR}})$ and 
then the de Rham comparison gives us an isomorphism of filtered $\mathbb{Q}_p$-vector spaces
$\mathbf{D}_{\mathrm{dR}}(V) \simeq \mathrm{H}^1_{\mathrm{dR}}(E/\mathbb{Q}_p)^\vee$
where $(-)^\vee$ is the dual with respect to the de Rham pairing.
Considering the Hodge filtration and its dual filtration,  we have
$\mathrm{Fil}^0\mathbf{D}_{\mathrm{dR}}(V) \simeq 
\left(\frac{\mathrm{H}^1_{\mathrm{dR}}(E/\mathbb{Q}_p)}{\mathrm{H}^0(E, \Omega^1)} \right)^\vee$
and
$\mathrm{dim}_{\mathbb{Q}_p}\mathrm{Fil}^0\mathbf{D}_{\mathrm{dR}}(V) = 1$.
For a $G_{\mathbb{Q}_p}$-equivariant endomorphism $\psi$ on $V$, the restriction of $\psi \otimes \mathrm{id}$ on $V \otimes \mathbf{B}_{\mathrm{dR}}$ to $\mathbf{D}_{\mathrm{dR}}(V)$ yields an endomorphism on $\mathbf{D}_{\mathrm{dR}}(V)$. 
Since the filtration comes from the $\mathbf{B}_{\mathrm{dR}}$-factor, this map preserves the filtration; thus, we have the endomorphism $\mathbf{D}^0_{\mathrm{dR}}(\psi)$ on $\mathrm{Fil}^0\mathbf{D}_{\mathrm{dR}}(V)$.
It defines a homomorphism of $\mathbb{Q}_p$-algebras
$\mathrm{End}_{ G_{\mathbb{Q}_p}} (V) \to \mathrm{End}_{\mathbb{Q}_p}\mathrm{Fil}^0\mathbf{D}_{\mathrm{dR}}(V)$
by sending $\psi$ to $\mathbf{D}^0_{\mathrm{dR}}(\psi)$.
Now suppose that $\psi$ is non-zero, so we also have the ($G_{\mathbb{Q}_p}$-equivariant) inverse $\psi^{-1}$.
The functoriality of Fontaine's de Rham functor yields that
$$( \psi \otimes \mathrm{id} ) \circ ( \psi^{-1} \otimes \mathrm{id} ) 
= ( \psi^{-1} \otimes \mathrm{id} ) \circ ( \psi \otimes \mathrm{id} ) 
= \mathrm{id}$$
on $\mathbf{D}_{\mathrm{dR}}(V)$.
Thus, $\mathbf{D}^0_{\mathrm{dR}}(\psi) \circ  \mathbf{D}^0_{\mathrm{dR}}(\psi^{-1}) = \mathbf{D}^0_{\mathrm{dR}}(\psi^{-1}) \circ \mathbf{D}^0_{\mathrm{dR}}(\psi)
= \mathrm{id}$
on $\mathrm{Fil}^0\mathbf{D}_{\mathrm{dR}}(V)$.
Therefore, the homomorphism 
$\mathrm{End}_{ G_{\mathbb{Q}_p}} (V) \to \mathrm{End}_{\mathbb{Q}_p}\mathrm{Fil}^0\mathbf{D}_{\mathrm{dR}}(V)$
is injective.
Since $\mathrm{Fil}^0\mathbf{D}_{\mathrm{dR}}(V)$ is one-dimensional, 
$\mathrm{End}_{\mathbb{Q}_p}\mathrm{Fil}^0\mathbf{D}_{\mathrm{dR}}(V) \simeq \mathbb{Q}_p$.
Thus, we have $\mathrm{End}_{ G_{\mathbb{Q}_p}} (V) \simeq \mathbb{Q}_p$.
\end{proof}
Recall the Cartier dual of $\mathrm{Sym}^2(T_pE)$:
$$\mathrm{Sym}^2(T_pE)^* := \mathrm{Hom}_{\mathbb{Z}_p}(\mathrm{Sym}^2(T_pE), \mathbb{Q}_p/\mathbb{Z}_p(1)) \simeq \mathrm{ad}^0(E[p^\infty]) .$$
Since the complex conjugation acts by $-1$ on $\mathbb{Q}_p/\mathbb{Z}_p(1)$, 
if a $\mathbb{Z}_p$-linear map
$\mathrm{Sym}^2(T_pE) \to \mathbb{Q}_p/\mathbb{Z}_p(1)$
is invariant under complex conjugation, it must vanish on $\mathrm{Sym}^2(T_pE)^+$ and can be arbitrary on $\mathrm{Sym}^2(T_pE)^-$.
Since $p > 2$, the $\pm$-parts split integrally. Thus, we have an isomorphism of $\mathbb{Z}_p$-modules
$\mathrm{H}^0(\mathbb{R}, \mathrm{ad}^0(E[p^\infty])) \simeq \mathrm{Hom}_{\mathbb{Z}_p}(\mathrm{Sym}^2(T_pE)^-, \mathbb{Q}_p/\mathbb{Z}_p) $, so we have
$$\mathrm{cork}_{\mathbb{Z}_p} \mathrm{H}^0(\mathbb{R}, \mathrm{ad}^0(E[p^\infty])) = \mathrm{rank}_{\mathbb{Z}_p} \mathrm{Sym}^2(T_pE)^- .$$
Thus, the notions of core ranks between \cite[Thm. 5.2.15]{mazur-rubin-book} and \cite[Thm. 5.4]{mazur-rubin-control} are compatible.
Since the complex conjugation has eigenvalues $+1$, $-1$  on \(T_pE\), so it has eigenvalues \(+1,-1,+1\) on 
$\mathrm{Sym}^2(T_pE)$; we have
$$\mathrm{rank}_{\mathbb{Z}_p} \mathrm{Sym}^2(T_pE)^- = 1 .$$
This confirms that the core rank $\chi(\mathrm{Sym}^2(T_pE), \mathcal{F}_{\mathrm{can}})$ is one.
\subsubsection{Bloch--Kato Selmer structure}
By using the global Poitou--Tate global duality computation (e.g. \cite[Proof of Thm. 5.2.15]{mazur-rubin-book}, \cite[Proof of Thm. 5.4]{mazur-rubin-control}) with the difference between $\mathcal{F}_{\mathrm{BK}}$ and $\mathcal{F}_{\mathrm{can}}$ (only at $p$), we need to compute $\mathrm{rank}_{\mathbb{Z}_p} \mathrm{H}^1_{/f}(\mathbb{Q}_p, \mathrm{Sym}^2(T_pE))$ where $\mathrm{H}^1_{f}$ means the Bloch--Kato Selmer structure at $p$.

The local Euler characteristic formula and local Tate duality yields
\begin{align*}
\mathrm{dim}_{\mathbb{Q}_p} \mathrm{H}^1(\mathbb{Q}_p, \mathrm{Sym}^2(V_pE))
& = 3 + \mathrm{dim}_{\mathbb{Q}_p} \mathrm{H}^0(\mathbb{Q}_p, \mathrm{Sym}^2(V_pE)) + \mathrm{cork}_{\mathbb{Q}_p} \mathrm{H}^0(\mathbb{Q}_p, \mathrm{ad}^0(E[p^\infty])) \\
& = 3 + \mathrm{dim}_{\mathbb{Q}_p} \mathrm{H}^0(\mathbb{Q}_p, \mathrm{Sym}^2(V_pE))
\end{align*}
where $V_pE = T_pE \otimes \mathbb{Q}_p$ and the last equality follows from Lemma \ref{lem:local-corank-zero}.

The fundamental exact sequence of Bloch--Kato yields the following formula \cite[Cor. 3.8.4]{bloch-kato}:
$$\mathrm{dim}_{\mathbb{Q}_p} \mathrm{H}^1_f(\mathbb{Q}_p, \mathrm{Sym}^2(V_pE))
= \mathrm{dim}_{\mathbb{Q}_p} \frac{\mathbf{D}_{\mathrm{dR}}(\mathrm{Sym}^2(V_pE))}{\mathrm{Fil}^0\mathbf{D}_{\mathrm{dR}}(\mathrm{Sym}^2(V_pE))} + \mathrm{dim}_{\mathbb{Q}_p} \mathrm{H}^0(\mathbb{Q}_p, \mathrm{Sym}^2(V_pE))$$
Since $\mathbf{D}_{\mathrm{dR}}(V_pE)$ has one-dimensional graded pieces in filtration degrees \(0\) and \(-1\), $\mathbf{D}_{\mathrm{dR}}(\mathrm{Sym}^2(V_pE))$ therefore has one-dimensional graded pieces in degrees \(0,-1,-2\). 
As a consequence, we have
$\mathrm{dim}_{\mathbb{Q}_p} \frac{\mathbf{D}_{\mathrm{dR}}(\mathrm{Sym}^2(V_pE))}{\mathrm{Fil}^0\mathbf{D}_{\mathrm{dR}}(\mathrm{Sym}^2(V_pE))} = 2$ and
$\mathrm{dim}_{\mathbb{Q}_p}  \mathrm{H}^1_{/f}(\mathbb{Q}_p, \mathrm{Sym}^2(V_pE)) =1$.
By pulling-back this computation to integral coefficients, we have
$$\mathrm{H}^1_{/f}(\mathbb{Q}_p, \mathrm{Sym}^2(T_pE)) \simeq \mathbb{Z}_p$$
since $\mathrm{H}^1_{/f}(\mathbb{Q}_p, \mathrm{Sym}^2(T_pE))$ is torsion-free.
Then the global Poitou--Tate duality computation comparing $\mathcal{F}_{\mathrm{BK}}$ with $\mathcal{F}_{\mathrm{can}}$ yields that the core rank $\chi(\mathrm{Sym}^2(T_pE), \mathcal{F}_{\mathrm{BK}})$ is zero.
The last computation is left to the reader.

\subsubsection{Bloch--Kato Selmer structure relaxed at a chosen prime $\ell$}
Let $\ell \neq p$ be a good reduction prime for $E$ such that 
the eigenvalues of $\mathrm{Fr}_\ell$ on $T_pE$ are $\alpha$ and $\alpha^{-1}\cdot \ell$ with $\alpha^2 \not\equiv 1 \pmod{p}$ as in Choice \ref{choice:tau-2}. 
Then the eigenvalues of $\mathrm{Fr}_\ell$ on $\mathrm{Sym}^2(T_pE)$ are $\alpha^2$, $\ell$, and $\alpha^{-2}\cdot \ell^2$.
By using the inflation-restriction sequence (e.g. \cite[Lem. 10.9]{weston-flach}),
we have isomorphism
$$ \mathrm{H}^1_{/f}(\mathbb{Q}_\ell, \mathrm{Sym}^2(V_pE)) 
\simeq  \left(\mathrm{Sym}^2(V_pE)(-1)\right)^{\mathrm{Fr}_\ell = 1}
\simeq  \left(\mathrm{ad}^0(V_pE)\right)^{\mathrm{Fr}_\ell = 1} \simeq \mathbb{Q}_p$$
since the eigenvalues of $\mathrm{Fr}_\ell$ on $\mathrm{ad}^0(V_pE)$ are 
$\alpha^2 \cdot \ell^{-1}$, 1, and $\alpha^{-2}\cdot \ell$.
This shows that
$$\mathrm{H}^1_{/f}(\mathbb{Q}_\ell, \mathrm{Sym}^2(T_pE)) \simeq \mathbb{Z}_p .$$
A similar global Poitou--Tate duality computation comparing $\mathcal{F}_{\mathrm{BK}}$ with $\mathcal{F}^{\ell}_{\mathrm{BK}}$ yields that the core rank $\chi(\mathrm{Sym}^2(T_pE), \mathcal{F}^{\ell}_{\mathrm{BK}})$ is one.
\begin{rem}
\begin{enumerate}
\item We do not need the $\ell \equiv 1 \pmod{p}$ condition in the last computation.
\item These computations confirm the core ranks of Kolyvagin systems described in \cite[Rem. 6.3.4]{mazur-rubin-book} and the core ranks of Kolyvagin systems and Stark systems are the same.
See also \cite[Thm. 4.1.13]{mazur-rubin-book} and \cite[Prop. 3.3]{mazur-rubin-control}.
\end{enumerate}
\end{rem}
\subsection{Verifying (H.7)}
For $q \in \mathcal{P}^{\mathrm{Flach}, (\ell)}_k$, we define
$$I_q := (1 - \alpha^2_q) \cdot(1 - q) \cdot (1 - \beta^2_q) = (1-q) \subseteq p^k \mathbb{Z}_p .$$
One may suitably choose $q$ in $\mathcal{P}^{\mathrm{Flach}, (\ell)}_k$ so that $I_q = p^k \mathbb{Z}_p$.
When we work over principal artinian ring $\mathbb{Z}/p^k\mathbb{Z}$, the image of $I_q$ in $\mathbb{Z}/p^k\mathbb{Z}$ becomes trivial.

\section{A proof of the explicit reciprocity law} \label{sec:explicit-reciprocity-law}
We recall the computation done by Weston in \cite{weston-flach} but with a very slight modification.
Proposition \ref{prop:explicit-reciprocity-law} follows from this computation.
\subsection{A lemma}
Let $q$ be a prime with $(q, Np)=1$.
Since $\mathrm{ad}^0(T_pE)$ is unramified at $q$, the action of $G_{\mathbb{Q}_q}$ on $\mathrm{ad}^0(T_pE)$ factors through its quotient by the inertia subgroup $\mathcal{I}_q$ at $q$.
We write
$\mathrm{H}^0(\mathbb{F}_q,  \mathrm{ad}^0(T_pE)) = \mathrm{H}^0(G_{\mathbb{Q}_q}/\mathcal{I}_q,  \mathrm{ad}^0(T_pE)) $.
\begin{lem} \label{lem:cartesian}
For a prime $q$ with $(q, Np)=1$,
we have
$$\mathrm{H}^1_{/f} (\mathbb{Q}_q , \mathrm{Sym}^2(T_pE)) \simeq \mathrm{H}^0(\mathbb{F}_q,  \mathrm{ad}^0(T_pE)) .$$
If $q$ is a $k$-Flach prime, then we have isomorphism
$$\psi_q:  \mathrm{H}^1_{/f} (\mathbb{Q}_q ,  \mathrm{Sym}^2(E[p^k])) \simeq \mathbb{Z}/p^k\mathbb{Z} .$$
\end{lem}
\begin{proof}
See \cite[Lem.  10.9]{weston-flach} for the first statement and the mod $p^k$ version also holds by the same reasoning. Since \cite[Lem.  10.9]{weston-flach} concerns the complex conjugation, we give the detail for the latter. 
Since $q$ is a $k$-Flach prime, 
we can choose a basis $x$, $y$ of $E[p^k]$ such that
$\mathrm{Fr}_q(x) = \alpha \cdot x$ and
$\mathrm{Fr}_q(y) = \alpha^{-1} \cdot q \cdot y$.
Then $x \otimes x $, $x\otimes y +  y\otimes x$, $y \otimes y$ forms a basis of $\mathrm{Sym}^2(E[p^k])$.
Since $\mathrm{Sym}^2(E[p^k])(-1) = \mathrm{ad}^0(E[p^k])$,
$\mathrm{Fr}_q$ acts on the induced basis of $\mathrm{ad}^0(E[p^k])$ by multiplication by $\alpha^2 \cdot q^{-1}$, $1$, and $\alpha^{-2} \cdot q$, respectively.
Since $q \equiv 1 \pmod{p^k}$ and $\alpha^2 \not\equiv 1 \pmod{p}$,
$ \mathrm{H}^1_{/f} (\mathbb{Q}_q ,  \mathrm{Sym}^2(E[p^k])) \simeq \mathrm{H}^0(\mathbb{F}_q,  \mathrm{ad}^0(E[p^k]))$ is free of rank one over $\mathbb{Z}/p^k\mathbb{Z}$, so we are done.
\end{proof}

\subsection{The computation of the image}
We recap Weston's computation of the image of $c(\ell)$ under $\mathrm{loc}^s_\ell$ but with Flach prime $\ell$. 
This computation slightly refines \cite[Lem. 2.5]{flach-annihilation}, and see \cite[Thm. 3.1.1]{weston-algebraic-cycles} for a more theoretical background of this computation.

As in Lemma \ref{lem:cartesian}, we choose a basis $x$, $y$ of $T_pE$ with respect to which $\mathrm{Fr}_\ell$ has matrix
$$ \begin{pmatrix}
\alpha & 0  \\
0 & \alpha^{-1} \cdot \ell
\end{pmatrix}$$
where $\alpha \in \mathbb{Z}_{p}$ with $\alpha^2 \not\equiv 1  \pmod{p}$.
By using the same idea of Lemma \ref{lem:cartesian} again,  we have isomorphism
\begin{align*}
\mathrm{H}^1_{/f}(\mathbb{Q}_\ell, \mathrm{Sym}^2(T_pE)) & \simeq \mathrm{H}^0(\mathbb{F}_\ell,  \mathrm{ad}^0(T_pE)) \\
& =  \mathbb{Z}_p \cdot
\begin{pmatrix}
1 & 0  \\
0 & -1
\end{pmatrix} .
\end{align*}
Following Weston's computation in \cite[p. 369]{weston-flach},  the image of $\mathrm{loc}^s_{\ell} c(\ell)$ in  $\mathrm{H}^0(\mathbb{F}_\ell,  \mathrm{ad}^0(T_pE))$
is
$$ 6 \cdot \mathrm{deg}(\phi) \cdot ( \alpha_\ell - \beta_\ell ) \cdot
\begin{pmatrix}
1 & 0  \\
0 & -1
\end{pmatrix}.$$
Since $p \geq 5$ and $\ell$ is a Flach prime (so $a_\ell(E) \not\equiv \pm 2 \pmod{p}$), $6 \cdot ( \alpha_\ell - \beta_\ell )$ is a $p$-adic unit.
Thus, we have an isomorphism of cyclic modules 
$$ \dfrac{ \mathrm{H}^1_{/f}(\mathbb{Q}_\ell, \mathrm{Sym}^2(T_pE)) }{  \mathbb{Z}_p \cdot \mathrm{loc}^s_{\ell} c(\ell) }  \simeq \mathbb{Z}_p / \mathrm{deg}(\phi) \mathbb{Z}_p ,$$
which implies the explicit reciprocity law (\ref{eqn:explicit-reciprocity-law}).  
\section{The construction of Flach--Stark systems} \label{sec:construction}
We freely use the language in \cite{mazur-rubin-book,mazur-rubin-control} in order to keep the argument concise.

\subsection{Basic setup for Stark systems} \label{subsec:choose_n}
\subsubsection{}
Fix a Flach prime $\ell$, and
let 
$$s = \mathrm{dim}_{\mathbb{F}_p} \mathrm{Sel}_{\ell\textrm{-}\mathrm{str} }( \mathbb{Q}, \mathrm{ad}^0(E[p])) $$
where $\mathrm{Sel}_{\ell\textrm{-}\mathrm{str} }$ means the $\ell$-strict Bloch--Kato Selmer structure.

Under our working hypotheses in \S\ref{sec:verifying-hypotheses}, for arbitrary $k \geq 1$, there exists 
\begin{equation} \label{eqn:choose_n}
n = q_1 \cdot q_2 \cdot \cdots \cdot q_s \in \mathcal{N}^{\mathrm{Flach},(\ell)}_k
\end{equation}
such that
$\mathrm{Sel}_{\ell n\textrm{-}\mathrm{str} }( \mathbb{Q}, \mathrm{ad}^0(E[p^k])) = 0$
where
 $\mathrm{Sel}_{\ell n\textrm{-}\mathrm{str} }$ means the  $\ell$-strict Bloch--Kato Selmer structure (Lemma \ref{lem:stark-systems-properties}(2) below).

Due to Theorem \ref{thm:flach-annihilation} and \cite[Lem. 3.5.3]{mazur-rubin-book}, we are able to and do fix an integer $k \gg 0$ such that
\begin{enumerate}
\item $\mathrm{Sel}( \mathbb{Q}, \mathrm{ad}^0(E[p^k])) = \mathrm{Sel}( \mathbb{Q}, \mathrm{ad}^0(E[p^\infty])) $, and
\item $(s+1) \cdot \mathrm{ord}_p \left( \mathrm{deg}(\phi) \right) < k$  
\end{enumerate}
throughout this section. In particular, the choice of $k$ is independent of $n$ in (\ref{eqn:choose_n}).
\begin{rem}
\begin{enumerate}
\item The $(s+1) \cdot \mathrm{ord}_p \left( \mathrm{deg}(\phi) \right) < k$ condition comes from equality
$$
\mathrm{ord}_p \left(
   \mathrm{loc}^s_\ell ( \mathrm{deg}(\phi)^s \cdot c(\ell) )
\right)
=
(s+1) \cdot \mathrm{ord}_p \left( \mathrm{deg}(\phi) \right)
$$
via Proposition \ref{prop:explicit-reciprocity-law} and it ensures that the outcome of the key computation in \S\ref{subsec:key-computation} is non-vacuous\footnote{We deeply thank an anonymous referee for pointing out this condition.}. See Remark \ref{rem:non-vanishing}.
\end{enumerate}
\end{rem}
\subsubsection{}
Following \cite[$\S$6]{mazur-rubin-control}, recall that
\begin{align*}
W_n & = \bigoplus^s_{i=1} \mathrm{Hom} \left( \mathrm{H}^1_{/f}( \mathbb{Q}_{q_i}, \mathrm{Sym}^2(E[p^k]) ), \mathbb{Z}/p^k\mathbb{Z} \right), \\
Y_n & = \bigwedge^{1+\nu(n)} \mathrm{Sel}_{\mathcal{F}^{\ell n}_{\mathrm{BK}}} ( \mathbb{Q}, \mathrm{Sym}^2(E[p^k]) )\otimes \bigwedge^{\nu(n)} W_n 
\end{align*}
where $\nu(n) =s$ is the number of prime divisors of $n$, and $\mathcal{F}^{\ell n}_{\mathrm{BK}}$ is the $\ell n$-relaxed Bloch--Kato Selmer structure.
For $m$ dividing $n$, we have the cartesian square
\[
\xymatrix{
\mathrm{Sel}_{\mathcal{F}^{\ell m}_{\mathrm{BK}}} ( \mathbb{Q}, \mathrm{Sym}^2(E[p^k]) ) \ar@{^{(}->}[r] \ar[d]^-{\oplus_{q \vert m} \mathrm{loc}^s_q} & \mathrm{Sel}_{\mathcal{F}^{\ell n}_{\mathrm{BK}}} ( \mathbb{Q}, \mathrm{Sym}^2(E[p^k]) ) \ar[d]^-{\oplus_{q \vert n} \mathrm{loc}^s_q} \\
\bigoplus_{q \vert m} \mathrm{H}^1_{/f}(\mathbb{Q}_q,\mathrm{Sym}^2(E[p^k]) )  \ar@{^{(}->}[r] &
\bigoplus_{q \vert n} \mathrm{H}^1_{/f}(\mathbb{Q}_q,\mathrm{Sym}^2(E[p^k]) ) 
}
\]
and the canonical map $\Psi_{n,m} : Y_n \to Y_m$ attached to the above square as in \cite[Prop. A.2]{mazur-rubin-control}.
For a $k$-Flach prime $q$, we have an isomorphism
$$\mathrm{H}^1_{/f} (\mathbb{Q}_q ,  \mathrm{Sym}^2(E[p^k])) \simeq \mathrm{H}^1_{\mathrm{tr}} (\mathbb{Q}_q ,  \mathrm{Sym}^2(E[p^k]))$$
by \cite[Lem. 1.2.4]{mazur-rubin-book} where $\mathrm{H}^1_{\mathrm{tr}} = \mathrm{H}^1_{L\textrm{-}\mathrm{tr}}$ means the ($L$-)transverse local condition with $L = \mathbb{Q}_q(\zeta_q)$ \cite[Def. 1.2, Lem. 1.3]{mazur-rubin-control}.
\subsubsection{}
We recall a proposition of Mazur--Rubin which plays a crucial role for the description of the canonical map $\Psi_{n,m}$ as well as the key computation in $\S$\ref{subsec:key-computation}.
\begin{prop}[Mazur--Rubin] \label{prop:mazur-rubin-unique-map}
Let $R$ be a local principal ideal ring with maximal ideal $\mathfrak{m}$.
Suppose that 
$$0 \to N \to M \xrightarrow{\psi} C$$ is an exact sequence of finitely generated $R$-modules, with $C$ free of rank one, and $r \geq 1$.
Then there exists a unique map
$$\widehat{\psi} : \wedge^{r} M  \to C \otimes \wedge^{r-1}N$$
such that
\begin{enumerate}
\item the composition 
$$\wedge^r M \xrightarrow{\widehat{\psi}} C \otimes \wedge^{r-1}N \to C \otimes \wedge^{r-1} M$$
is given by 
$$m_1 \wedge \cdots \wedge m_r \mapsto \sum^{r}_{i=1} (-1)^{i+1} \cdot \psi(m_i) \otimes (m_1 \wedge \cdots \wedge m_{i-1} \wedge m_{i+1} \wedge \cdots \wedge m_r) ,$$
\item the image of $\widehat{\psi}$ is the image of $\psi(M) \otimes \wedge^{r-1}N \to C \otimes \wedge^{r-1}N$.
\end{enumerate}
If $M$ is free of rank $r$ over $R$, then $\widehat{\psi}$ is an isomorphism if and only if $\psi$ is surjective.
\end{prop}
\begin{proof}
See \cite[Prop. A.1]{mazur-rubin-control}.
\end{proof}
Fix
\begin{itemize}
\item  a factorization $n = q_1 \cdots q_t$ with $m = q_1 \cdots q_s$, and
\item  a generator $\psi_{q_i} \in \mathrm{Hom}(\mathrm{H}^1_{\mathrm{tr}} (\mathbb{Q}_{q_i} ,  \mathrm{Sym}^2(E[p^k])), \mathbb{Z}/p^k\mathbb{Z})$ for every $i$.
\end{itemize}
Let $n_i = \prod_{j \leq i} q_j$.
This choice yields a map
$$\widehat{ \psi_{q_{s+1}} \circ \mathrm{loc}^{\mathrm{tr}}_{q_{s+1}} } 
\circ \cdots \circ
  \widehat{ \psi_{q_{t}} \circ \mathrm{loc}^{\mathrm{tr}}_{q_{t}} }
  : \wedge^{1+t} \mathrm{H}^1_{\mathcal{F}^{\ell n}_{\mathrm{BK}}}(\mathbb{Q}, \mathrm{Sym}^2(E[p^k]))
\to \wedge^{1+s} \mathrm{H}^1_{\mathcal{F}^{\ell m}_{\mathrm{BK}}}(\mathbb{Q}, \mathrm{Sym}^2(E[p^k]))$$
where 
$$\widehat{ \psi_{q_{i}} \circ \mathrm{loc}^{\mathrm{tr}}_{q_{i}} } : \wedge^i \mathrm{H}^1_{\mathcal{F}^{\ell n_i}_{\mathrm{BK}}}(\mathbb{Q}, \mathrm{Sym}^2(E[p^k])) \to
\wedge^{i-1} \mathrm{H}^1_{\mathcal{F}^{\ell n_{i-1}}_{\mathrm{BK}}}(\mathbb{Q}, \mathrm{Sym}^2(E[p^k]))$$
comes from Proposition \ref{prop:mazur-rubin-unique-map}
and 
$\mathrm{loc}^{\mathrm{tr}}_{q_{i}} : \mathrm{H}^1_{\mathcal{F}^{\ell n_i}_{\mathrm{BK}}}(\mathbb{Q}, \mathrm{Sym}^2(E[p^k])) \to \mathrm{H}^1_{\mathrm{tr}} (\mathbb{Q}_{q_i} ,  \mathrm{Sym}^2(E[p^k])) $.
The choice also yields an isomorphism
$$\wedge^{\nu(n)} W_n \simeq  \wedge^{\nu(m)} W_m$$
sending
$\psi_{q_{1}} \wedge \cdots \wedge \psi_{q_{t}}$
to $\psi_{q_{1}} \wedge \cdots \wedge \psi_{q_{s}}$.
Then $\Psi_{n,m} : Y_n \to Y_m$ is defined by the tensor product of these two maps and it is independent of the choices we made \cite[p. 161]{mazur-rubin-control}.
\subsubsection{}
The canonical map $\Psi_{n,m} $ has the following compatibility.
\begin{prop} \label{prop:compatibility}
Suppose $n_0 \in \mathcal{N}^{\mathrm{Flach},(\ell)}_k$, $n_1$ divides $n_0$, and $n_2$ divides $n_1$.
Then
$$\Psi_{n_0,n_2} = \Psi_{n_1,n_2} \circ \Psi_{n_0,n_1} .$$
\end{prop}
\begin{proof}
See \cite[Prop. 6.4]{mazur-rubin-control}.
\end{proof}
\subsection{Definition and properties of Stark systems}
We briefly recall the definition of Stark systems \cite[Def. 6.5, Def. 7.1]{mazur-rubin-control} and their important properties.
\begin{defn} \label{defn:stark-systems}
\begin{enumerate}
\item 
The \textbf{$\mathbb{Z}/p^k\mathbb{Z}$-module of the (rank one) Stark system} 
$$\mathbf{SS}_1(\mathrm{Sym}^2(E[p^k])) = \mathbf{SS}_1(\mathrm{Sym}^2(E[p^k]), \mathcal{F}^{\ell}_{\mathrm{BK}}, \mathcal{P}^{\mathrm{Flach}, (\ell)}_k)$$
for the Selmer data $(\mathrm{Sym}^2(E[p^k]), \mathcal{F}^{\ell}_{\mathrm{BK}}, \mathcal{P}^{\mathrm{Flach}, (\ell)}_k)$
is defined to be the inverse limit
$$\mathbf{SS}_1(\mathrm{Sym}^2(E[p^k])) = \varprojlim_{n_0 \in \mathcal{N}^{\mathrm{Flach}, (\ell)}_k} Y_{n_0}$$
with respect to the maps $\Psi_{n_0,n_1}$ with $n_1 \vert n_0$.
\item 
The \textbf{$\mathbb{Z}_p$-module of the (rank one) Stark system} 
$$\mathbf{SS}_1(\mathrm{Sym}^2(T_pE)) = \mathbf{SS}_1(\mathrm{Sym}^2(T_pE), \mathcal{F}^{\ell}_{\mathrm{BK}}, \mathcal{P}^{(\ell)})$$
for the Selmer data $(\mathrm{Sym}^2(T_pE), \mathcal{F}^{\ell}_{\mathrm{BK}}, \mathcal{P}^{(\ell)})$
is defined to be the inverse limit
$$\mathbf{SS}_1(\mathrm{Sym}^2(T_pE)) = 
\varprojlim_{n_0 \in \mathcal{N}^{\mathrm{Flach}, (\ell)}_k} Y_{n_0}[I_{n_0}]$$
with respect to the maps $\Psi_{n_0,n_1} \pmod{I_{n_1}}$ with $n_1 \vert n_0$
where 
$$Y_{n_0}[I_{n_0}]  = \bigwedge^{1+\nu(n_0)} \mathrm{Sel}_{\mathcal{F}^{\ell n_0}_{\mathrm{BK}}} ( \mathbb{Q}, \mathrm{Sym}^2(E[I_{n_0}]) )\otimes \bigwedge^{\nu(n_0)} \bigoplus^s_{i=1} \mathrm{Hom} \left( \mathrm{H}^1_{/f}( \mathbb{Q}_{q_i}, \mathrm{Sym}^2(E[I_{n_0}]) ), \mathbb{Z}_p/I_{n_0}\mathbb{Z}_p \right)$$
and
$I_{n_0} = \sum_{q \vert n_0} I_{q} \subseteq \mathbb{Z}_p$ following (H.7) in \S\ref{sec:verifying-hypotheses}.
\item
An element $\stark^{\ell} = \varprojlim_{n_0 \in \mathcal{N}^{\mathrm{Flach}, (\ell)}_k} \epsilon^{\ell}_{n_0}$ in $\mathbf{SS}_1(\mathrm{Sym}^2(E[p^k]))$ or $\mathbf{SS}_1(\mathrm{Sym}^2(T_pE))$ is called a \textbf{Stark system}.
\end{enumerate}
\end{defn}
\begin{prop}
The natural maps $\mathrm{Sym}^2(T_pE) \twoheadrightarrow \mathrm{Sym}^2(E[p^k])$ and $\mathcal{P}^{\mathrm{Flach}, (\ell)}_k \hookrightarrow \mathcal{P}^{(\ell)}$
induce an isomorphism
$$\mathbf{SS}_1(\mathrm{Sym}^2(T_pE)) \simeq \varprojlim_{k} \mathbf{SS}_1(\mathrm{Sym}^2(E[p^k]))$$
where the inverse limit is with respect to the maps
$$\mathbf{SS}_1(\mathrm{Sym}^2(E[p^k]), \mathcal{P}^{\mathrm{Flach}, (\ell)}_k) \twoheadrightarrow
\mathbf{SS}_1(\mathrm{Sym}^2(E[p^j]),  \mathcal{P}^{\mathrm{Flach}, (\ell)}_k)
\xleftarrow{\simeq}
\mathbf{SS}_1(\mathrm{Sym}^2(E[p^j]),  \mathcal{P}^{\mathrm{Flach}, (\ell)}_j)$$
with $j \leq k$.
\end{prop}
\begin{proof}
See \cite[Lem. 7.2, Prop. 7.3]{mazur-rubin-control}.
\end{proof}
For $\mathcal{N}^{\mathrm{Flach},(\ell)}_k$, let 
$$Y'_n = 
p^{\mathrm{length}(  \mathrm{Sel}_{\ell n\textrm{-}\mathrm{str} }(\mathbb{Q}, \mathrm{ad}^0(E[p^k])) ) } Y_n .$$
\begin{lem} \label{lem:stark-systems-properties}
Under our working hypotheses in \S\ref{sec:verifying-hypotheses}, we have the following statements:
\begin{enumerate}
\item $Y'_n$ is a cyclic $\mathbb{Z}/p^k\mathbb{Z}$-module of length $\mathrm{max} \lbrace k - \mathrm{length}(  \mathrm{Sel}_{\mathcal{F}_{\mathrm{BK},\ell n}}(\mathbb{Q}, \mathrm{ad}^0(E[p^k])) )  , 0 \rbrace$.
\item There exist $n \in \mathcal{N}^{\mathrm{Flach},(\ell)}_k$ such that $\mathrm{Sel}_{\ell n\textrm{-}\mathrm{str} }(\mathbb{Q}, \mathrm{ad}^0(E[p^k])) = 0$.
\item If $\mathrm{Sel}_{\ell n\textrm{-}\mathrm{str} }(\mathbb{Q}, \mathrm{ad}^0(E[p^k])) = 0$, then $Y_n$ is free of rank one over $\mathbb{Z}/p^k\mathbb{Z}$.
\item If $\mathrm{Sel}_{\ell n\textrm{-}\mathrm{str} }(\mathbb{Q}, \mathrm{ad}^0(E[p^k])) = 0$ and $m$ divides $n$, then $\Psi_{n,m} (Y_n) = Y'_m$.
\end{enumerate}
\end{lem}
\begin{proof}
See \cite[Lem. 6.6]{mazur-rubin-control}.
\end{proof}
\begin{thm} \label{thm:core-rank-one}
We keep our working hypotheses in \S\ref{sec:verifying-hypotheses}. Then the following statements are valid:
\begin{enumerate}
\item
$\mathbf{SS}_1(\mathrm{Sym}^2(E[p^k])) \simeq \mathbb{Z}/p^k \mathbb{Z}$
and the image of the projection map $\mathbf{SS}_1(\mathrm{Sym}^2(E[p^k])) \to Y_n$ is $Y'_n$.
\item $\mathbf{SS}_1(\mathrm{Sym}^2(T_pE)) \simeq \mathbb{Z}_p$
and it is generated by a Stark system whose image in $\mathbf{SS}_1(\mathrm{Sym}^2(E[p]))$ is non-zero.
\end{enumerate}
\end{thm}
A Stark system \textbf{$\stark^{\ell}$ is primitive} if the image of $\stark^{\ell}$ in $\mathbf{SS}_1(\mathrm{Sym}^2(E[p]))$ is non-zero \cite[\S8]{mazur-rubin-control}.
\begin{proof}
See \cite[Thm. 6.7, Thm. 7.4]{mazur-rubin-control}.
\end{proof}
Let $\stark^{\ell}$ be a Stark system for $\mathrm{Sym}^2(T_pE)$. 
Following \cite[\S8]{mazur-rubin-control}, we write
\begin{align*}
\mathrm{ord}_p (\epsilon^{\ell}_n) & : = \mathrm{max} \lbrace j : \epsilon^{\ell}_n \in p^j Y_n \rbrace,  n \in \mathcal{N}^{\mathrm{Flach}, (\ell)}_1 , \\
\partial\varphi_{\stark^{\ell}}(\infty) & := \lim_{t \to \infty} \partial\varphi_{\stark^{\ell}}(t) \\
& := 
\lim_{t \to \infty} \mathrm{min} \lbrace \mathrm{ord}_p (\epsilon^{\ell}_n) : n \in \mathcal{N}^{\mathrm{Flach}, (\ell)}_1 , \nu(n) = t \rbrace \\
& = 
\mathrm{min} \lbrace \mathrm{ord}_p (\epsilon^{\ell}_n) : n \in \mathcal{N}^{\mathrm{Flach}, (\ell)}_1 \rbrace .
\end{align*}
The theory of Stark systems yields the following result on the exact size of the dual Selmer groups.
\begin{thm} \label{thm:structure}
We keep our working hypotheses in \S\ref{sec:verifying-hypotheses}, and let $\stark^{\ell}$ be a Stark system for $\mathrm{Sym}^2(T_pE)$. 
\begin{enumerate}
\item 
If $\stark^{\ell}$ is non-trivial, then
$$\mathrm{length}_{\mathbb{Z}_p} \mathrm{Sel}_{\ell\textrm{-}\mathrm{str}}(\mathbb{Q}, \mathrm{ad}^0(E[p^\infty])) = \mathrm{ord}_p (\epsilon^{\ell}_1) - 
\partial\varphi_{\stark^{\ell}}(\infty) .$$
\item 
If $\stark^{\ell}$ is primitive, then 
$$\mathrm{length}_{\mathbb{Z}_p} \mathrm{Sel}_{\ell\textrm{-}\mathrm{str}}(\mathbb{Q}, \mathrm{ad}^0(E[p^\infty])) = \mathrm{ord}_p (\epsilon^{\ell}_1).$$
\item 
$\stark^{\ell}$ is primitive if and only if $\partial\varphi_{\stark^{\ell}}(\infty) = 0$.
\end{enumerate}\end{thm}
\begin{proof}
See \cite[Thm. 8.7]{mazur-rubin-control}. 
Under the non-triviality of $\stark^{\ell}$, the non-triviality of $\epsilon^{\ell}_1$ is equivalent to the finiteness of $\mathrm{Sel}_{\ell\textrm{-}\mathrm{str}}(\mathbb{Q}, \mathrm{ad}^0(E[p^\infty]))$, which follows from Theorem \ref{thm:flach-annihilation}.
\end{proof}
\subsection{The key computation} \label{subsec:key-computation}
We keep the choices of $\ell$ and $k$ in \S\ref{subsec:choose_n}. 
\subsubsection{} \label{subsubsec:pre-computation}
Write
$$n = q_1 \cdot q_2 \cdot \cdots \cdot q_s \in \mathcal{N}^{\mathrm{Flach},(\ell)}_k$$
as in \eqref{eqn:choose_n}.
Let $c(\ell)$ and $c(q_i)$ be Flach's zeta elements.
Then we have
\begin{itemize}
\item $\mathrm{loc}^{s}_{q_i}(c(r))=0$ in $\mathrm{H}^1_{/f} (\mathbb{Q}_{q_i} , \mathrm{Sym}^2(T_pE))$ for every $r \neq q_i$ (Theorem \ref{thm:flach-zeta-elements}), and
\item  $\mathrm{loc}^{s}_{r}(c(r))$ generates $\mathrm{deg}(\phi) \cdot \mathrm{H}^1_{/f} (\mathbb{Q}_{r} , \mathrm{Sym}^2(T_pE))$ for a Flach prime $r$ (Proposition \ref{prop:explicit-reciprocity-law}).
\end{itemize}

\subsubsection{}
For the same $n \in \mathcal{N}^{\mathrm{Flach},(\ell)}_k$, we explicitly compute 
$$\Psi_{n,n/q_s} \left( c(\ell) \wedge \bigwedge^s_{i=1} c(q_i) \otimes \bigwedge^s_{i=1} \psi_{q_i} \right)$$
where $\psi_{q_i} : \mathrm{H}^1_{/f} (\mathbb{Q}_{q_i} , \mathrm{Sym}^2(E[p^k])) \simeq \mathbb{Z}/p^k\mathbb{Z}$ is the map in Lemma \ref{lem:cartesian}.
Here we use the same notation $c(\ell)$ and $c(q_i)$ for their image  in $\mathrm{H}^1 ( \mathbb{Q}, \mathrm{Sym}^2(E[p^k]) )$.

Consider exact sequence
\[
\xymatrix{
 \mathrm{Sel}_{\mathcal{F}^{\ell n/q_s}_{\mathrm{BK}}}(\mathbb{Q}, \mathrm{Sym}^2(E[p^k])) \ar@{^{(}->}[r]  & \mathrm{Sel}_{\mathcal{F}^{\ell n}_{\mathrm{BK}}}(\mathbb{Q}, \mathrm{Sym}^2(E[p^k])) \ar[rr]^-{ \psi_{q_s} \circ \mathrm{loc}^s_{q_s}} & & \mathbb{Z}/p^k\mathbb{Z} .
}
\]
By applying Proposition \ref{prop:mazur-rubin-unique-map} to the above sequence, we obtain a unique map 
$$\widehat{\psi_{q_s} \circ \mathrm{loc}^s_{q_s}} : \bigwedge^{s+1} \mathrm{Sel}_{\mathcal{F}^{\ell n}_{\mathrm{BK}}}(\mathbb{Q}, \mathrm{Sym}^2(E[p^k])) \to \mathbb{Z}/p^k\mathbb{Z} \otimes \bigwedge^{s} \mathrm{Sel}_{\mathcal{F}^{\ell n /q_s }_{\mathrm{BK}}}(\mathbb{Q}, \mathrm{Sym}^2(E[p^k]))$$
such that the composition
\[
\xymatrix{
\bigwedge^{s+1} \mathrm{Sel}_{\mathcal{F}^{\ell n}_{\mathrm{BK}}}(\mathbb{Q}, \mathrm{Sym}^2(E[p^k])) \ar[rr]^-{ \widehat{\psi_{q_s} \circ \mathrm{loc}^s_{q_s}} } & & \mathbb{Z}/p^k\mathbb{Z} \otimes \bigwedge^{s} \mathrm{Sel}_{\mathcal{F}^{\ell n / q_s}_{\mathrm{BK}}}(\mathbb{Q}, \mathrm{Sym}^2(E[p^k])) \ar[d] \\
& & \mathbb{Z}/p^k\mathbb{Z} \otimes \bigwedge^{s} \mathrm{Sel}_{\mathcal{F}^{\ell n}_{\mathrm{BK}}}(\mathbb{Q}, \mathrm{Sym}^2(E[p^k]))
}
\]
is given by
$$c(\ell) \wedge \bigwedge_{i=1}^{s} c(q_i) \mapsto 
\psi_{q_s} \circ \mathrm{loc}^s_{q_s} (c(\ell)) \cdot \bigwedge_{i=1}^{s} c(q_i) +
\sum_{i=1}^{s} \left( (-1)^{i} \cdot \psi_{q_s} \circ \mathrm{loc}^s_{q_s} (c(q_i)) \cdot c(\ell) \wedge \bigwedge^s_{ \substack{ j=1 \\ j \neq i}} c(q_j) \right) ,$$
and the image of $\widehat{\psi_{q_s} \circ \mathrm{loc}^s_{q_s}}$ is the image of
$\psi_{q_s} \circ \mathrm{loc}^s_{q_s} \left( \mathrm{Sel}_{\mathcal{F}^{\ell n}_{\mathrm{BK}}}(\mathbb{Q}, \mathrm{Sym}^2(E[p^k])) \right) \otimes \bigwedge^{s} \mathrm{Sel}_{\mathcal{F}^{\ell n / q_s}_{\mathrm{BK}}}(\mathbb{Q}, \mathrm{Sym}^2(E[p^k]))$ in $\mathbb{Z}/p^k\mathbb{Z} \otimes \bigwedge^{s} \mathrm{Sel}_{\mathcal{F}^{\ell n / q_s}_{\mathrm{BK}}}(\mathbb{Q}, \mathrm{Sym}^2(E[p^k])) $.
\subsubsection{}
By using the computation of the localizations of Flach's zeta elements in $\S$\ref{subsubsec:pre-computation},
we have
\begin{align*}
&  \psi_{q_s} \circ \mathrm{loc}^s_{q_s} (c(\ell)) \cdot \bigwedge_{i=1}^{s} c(q_i) +
\sum_{i=1}^{s} \left( (-1)^{i} \cdot \psi_{q_s} \circ\mathrm{loc}^s_{q_s} (c(q_i)) \cdot c(\ell) \wedge \bigwedge^s_{ \substack{ j=1 \\ j \neq i}} c(q_j) \right) \\
& = (- 1)^s \cdot\psi_{q_s} \circ \mathrm{loc}^s_{q_s} (c(q_s)) \cdot c(\ell) \wedge \bigwedge^{s-1}_{i=1} c(q_i) .
\end{align*}
Following the concrete description of $\Psi_{n,m}$ in \cite[p. 161]{mazur-rubin-control}, we have
\begin{equation} \label{eqn:key-computation}
\Psi_{n,n/q_s} \left( c(\ell) \wedge \bigwedge^s_{i=1} c(q_i) \otimes \bigwedge^s_{i=1} \psi_{q_i} \right) =  (- 1)^s \cdot \psi_{q_s} \circ \mathrm{loc}^s_{q_s} (c(q_s)) \cdot c(\ell) \wedge \bigwedge^{s-1}_{i=1} c(q_i) \otimes \bigwedge^{s-1}_{i=1} \psi_{q_i} .
\end{equation}
The computation of the image under $\Psi_{n,n/q_i}$ is identical for every $i$ possibly except the sign.
By Proposition \ref{prop:compatibility}, we also have
\begin{equation} \label{eqn:composition-stark-system-map}
\Psi_{n,1} = \Psi_{q_1, 1} \circ \Psi_{q_1\cdot q_2 , q_1} \circ  \cdots \circ \Psi_{n/q_s, n/(q_{s-1}\cdot q_{s})} \circ \Psi_{n, n/q_s} .
\end{equation}
\subsection{A Stark/Kolyvagin system of Gauss sum type} \label{subsec:stark-systems-gauss-sum-type}
\subsubsection{}
We keep the choices of $\ell$ and $k$ in \S\ref{subsec:choose_n} as before.
For each $n \in \mathcal{N}^{\mathrm{Flach},(\ell)}_k$,
define
$$ \epsilon^{\mathrm{Flach}, \ell}_n := (- 1)^{\frac{s(s+1)}{2}} \cdot  c(\ell) \wedge \bigwedge^s_{i=1} c(q_i) \otimes \bigwedge^s_{i=1} \psi_{q_i}  \in Y_n $$
so that
\begin{align*}
\Psi_{n,1} ( \epsilon^{\mathrm{Flach}, \ell}_n ) & = \left( \prod_{i=1}^s \psi_{q_i} \circ \mathrm{loc}^s_{q_i}(c(q_i)) \right) \cdot c(\ell) \\
& =  \mathrm{deg}(\phi)^s  \cdot c(\ell)
\end{align*}
This formula follows from the iteration of (\ref{eqn:key-computation}) via (\ref{eqn:composition-stark-system-map}) and Proposition \ref{prop:explicit-reciprocity-law}.
In particular, the $(- 1)^{\frac{s(s+1)}{2}}$ term comes from the $(- 1)^s$ term in \eqref{eqn:key-computation}.
\begin{rem} \label{rem:non-vanishing}
Due to the second condition of $k$ in \S\ref{subsec:choose_n}, $\mathrm{deg}(\phi)^s  \cdot c(\ell)$ does not vanish in $\mathrm{Sel}_{\mathcal{F}^{\ell}_{\mathrm{BK}}} ( \mathbb{Q}, \mathrm{Sym}^2(E[p^k]) )$. This ensures the non-triviality of the Stark system which we construct here.
\end{rem}
For $m$ dividing $n$, we define
$$\epsilon^{\mathrm{Flach}, \ell}_m := \Psi_{n,m}(\epsilon^{\mathrm{Flach}, \ell}_n) .$$
By Definition \ref{defn:stark-systems},
$ \left\lbrace \epsilon^{\mathrm{Flach}, \ell}_m \in Y_m : m \vert n \right\rbrace$
forms a \emph{finite} Stark system of core rank one,
i.e. a Stark system for the Selmer data $\left(\mathrm{Sym}^2(E[p^k]), \mathcal{F}^{\ell}_{\mathrm{BK}}, \mathcal{P}^{n} \right)$
where $\mathcal{P}^{n} $ is the (finite!) set of the primes dividing $n$.
We call it the finite Flach--Stark system.
\begin{rem}
In the case of Kolyvagin systems of Gauss sums \cite[Rem. 6.3.3]{mazur-rubin-book}, there is a restriction of the choice of $n$ which depends seriously on $\ell$. Namely, every prime divisor of $n$ should divide $\ell -1$.
\end{rem}
\subsubsection{}
Let $m$ be a divisor of $n$ and $G_m = \otimes_{q \vert m} \mathrm{Gal}(\mathbb{Q}(\zeta^{(p)}_{q})/\mathbb{Q})$ where
$\mathbb{Q}(\zeta^{(p)}_{q})$ is the maximal $p$-subextension of $\mathbb{Q}$ in $\mathbb{Q}(\zeta_{q})$.
Recall the map in \cite[$\S$12]{mazur-rubin-control}
 \begin{equation} \label{eqn:stark-to-kolyvagin}
 \Pi_m : Y_m \to \mathrm{Sel}_{\mathcal{F}^{\ell}_{\mathrm{BK}}(m)} (\mathbb{Q}, \mathrm{Sym}^2(E[p^k]) ) \otimes G_m
 \end{equation}
where
$\mathcal{F}^{\ell}_{\mathrm{BK}}(m)$ is the $\ell$-relaxed and $m$-transverse Bloch--Kato Selmer structure.
Then \cite[Prop. 12.3]{mazur-rubin-control} implies that
$$ \left\lbrace (-1)^{\nu(m)} \cdot \Pi_m ( \epsilon^{\mathrm{Flach}, \ell}_m)  : m \vert n \right\rbrace$$
forms a \emph{finite} Kolyvagin system for $\left( \mathrm{Sym}^2(E[p^k]), \mathcal{F}^{\ell}_{\mathrm{BK}}, \mathcal{P}^n \right)$.
It is natural to call it the finite Flach--Kolyvagin system.

\subsection{An extension to a genuine Stark/Kolyvagin system} \label{subsec:extension-to-genuine-systems}
We explain how a Stark/Kolyvagin system of Gauss sum type extends to a genuine system via the rigidity of Stark/Kolyvagin systems. This gives a detailed account of \cite[Rem. 6.3.3]{mazur-rubin-book}.
\subsubsection{}
For every $n \in \mathcal{N}^{\mathrm{Flach},(\ell)}_k$ with $\nu(n)=s$, we have the (non-trivial) finite Flach--Stark system 
$$ \left\lbrace \epsilon^{\mathrm{Flach}, \ell}_m \in Y_m : m \vert n \right\rbrace$$ as we constructed in \S\ref{subsec:stark-systems-gauss-sum-type}.

We now suppose that $n \in \mathcal{N}^{\mathrm{Flach},(\ell)}_k$ satisfies $\mathrm{Sel}_{\ell n\textrm{-}\mathrm{str} }( \mathbb{Q}, \mathrm{ad}^0(E[p^k])) = 0$.
Such an $n$ always exists due to Lemma \ref{lem:stark-systems-properties}(2).
In this situation, we have $Y'_n = Y_n$ by definition and it is free of rank one over $\mathbb{Z}/p^k\mathbb{Z}$ (Lemma \ref{lem:stark-systems-properties}(3)).
Thus Theorem \ref{thm:core-rank-one}(1) yields that the projection map
$$\mathbf{SS}_1(\mathrm{Sym}^2(E[p^k])) \xrightarrow{\simeq} Y'_n = Y_n$$
is an isomorphism of free $\mathbb{Z}/p^k\mathbb{Z}$-modules of rank one.
Then the inverse image of $\epsilon^{\mathrm{Flach}, \ell}_n \in Y_n$ under the projection map is the genuine Stark system extending $\epsilon^{\mathrm{Flach}, \ell}_n$. In particular, it extends the finite Flach--Stark system $ \left\lbrace \epsilon^{\mathrm{Flach}, \ell}_m \in Y_m : m \vert n \right\rbrace$.
We call the extended Stark system $\stark^{\mathrm{Flach}, \ell}$  for $\left( \mathrm{Sym}^2(E[p^k]), \mathcal{F}^{\ell}_{\mathrm{BK}}, \mathcal{P}^{\mathrm{Flach},(\ell)}_k \right)$ the \textbf{Flach--Stark system}.
\begin{rem}
This $n$ plays the role of a core vertex in the theory of Kolyvagin systems.
\end{rem}
\subsubsection{}
By using the map (\ref{eqn:stark-to-kolyvagin})
$$ \Pi_n : Y_n \to \mathrm{Sel}_{\mathcal{F}^{\ell}_{\mathrm{BK}}(n)} (\mathbb{Q}, \mathrm{Sym}^2(E[p^k]) ) \otimes G_n$$
for every $n \in \mathcal{N}^{\mathrm{Flach},(\ell)}_k$,
there exists an isomorphism between free $\mathbb{Z}/p^k\mathbb{Z}$-modules of rank one
$$\mathbf{SS}_1(\mathrm{Sym}^2(E[p^k])) \xrightarrow{\simeq} \mathbf{KS}_1(\mathrm{Sym}^2(E[p^k]))$$
where $\mathbf{KS}_1(\mathrm{Sym}^2(E[p^k]))$ is the module of Kolyvagin systems for $\left( \mathrm{Sym}^2(E[p^k]), \mathcal{F}^{\ell}_{\mathrm{BK}}, \mathcal{P}^{\mathrm{Flach},(\ell)}_k \right)$.
It follows from  \cite[Prop. 12.3, Thm. 12.4]{mazur-rubin-control} with help of \cite[Thm. 4.4.1]{mazur-rubin-book} and \cite[Rem. 11.5]{mazur-rubin-control}.
The (extended) \textbf{Flach--Kolyvagin system}
$\ks^{\mathrm{Flach}, \ell}$ for $\left( \mathrm{Sym}^2(E[p^k]), \mathcal{F}^{\ell}_{\mathrm{BK}}, \mathcal{P}^{\mathrm{Flach},(\ell)}_k \right)$
is defined by the image of $\stark^{\mathrm{Flach}, \ell}$ under the isomorphism above.
Then $\ks^{\mathrm{Flach}, \ell}$ extends the finite Flach--Kolyvagin system.

\begin{rem}
\begin{enumerate}
\item The second condition of $k$ in \S\ref{subsec:choose_n} can be lifted for the (extended) Flach--Stark system $\stark^{\mathrm{Flach}, \ell}$ and the (extended) Flach--Kolyvagin system $\ks^{\mathrm{Flach}, \ell}$.
\item If $n \in \mathcal{P}^{\mathrm{Flach},(\ell)}_k$ with $\mathrm{Sel}_{\ell n\textrm{-}\mathrm{str}} (\mathbb{Q}, \mathrm{ad}^0(E[p^k]) ) =0$, then
$\kappa^{\mathrm{Flach},\ell}_n$ is uniquely determined by $\kappa^{\mathrm{Flach}, \ell}_m$ for $m$ properly dividing $n$ \cite[Rem. 3.1.9]{mazur-rubin-book}.
From this point of view, the extension of a finite Kolyvagin system is not a miracle and it shows the strength of the rigidity of Kolyvagin systems. 
\end{enumerate}
\end{rem}

\subsection{The constant multiple} \label{subsec:optimal-normalization}
We compare $\stark^{\mathrm{Flach}, \ell}$ constructed in \S\S\ref{subsec:stark-systems-gauss-sum-type}-\ref{subsec:extension-to-genuine-systems} with the \emph{primitive} one and compute the constant multiple in Theorem \ref{thm:flach-stark}.

Under the local Tate pairing, the global Poitou--Tate duality \cite[Thm. 1.7.3]{rubin-book}, \cite[Thm. 2.3.4]{mazur-rubin-book} identifies $\frac{\mathrm{Sel}( \mathbb{Q}, \mathrm{ad}^0(E[p^k]))}{\mathrm{Sel}_{ \ell\textrm{-}\mathrm{str} }( \mathbb{Q}, \mathrm{ad}^0(E[p^k]))}$
with the cokernel of the singular quotient of the localization of $\mathrm{Sel}_{ \mathcal{F}^{\ell}_{\mathrm{BK}} }( \mathbb{Q}, \mathrm{Sym}^2(E[p^k]))$ at $\ell$ with the Pontryagin dual
\begin{equation} \label{eqn:global-duality-one-ell}
{\footnotesize
\xymatrix{
\mathrm{Sel}_{ \ell\textrm{-}\mathrm{str} }( \mathbb{Q}, \mathrm{ad}^0(E[p^k])) \ar@{^{(}->}[r] &
\mathrm{Sel}( \mathbb{Q}, \mathrm{ad}^0(E[p^k])) \ar@{->>}[r] &
\left(
\dfrac{ \mathrm{H}^1_{/f}(\mathbb{Q}_\ell , \mathrm{Sym}^2(E[p^k]) )  }{ \mathrm{loc}^s_\ell \left( \mathrm{Sel}_{ \mathcal{F}^{\ell}_{\mathrm{BK}} }( \mathbb{Q}, \mathrm{Sym}^2(E[p^k]))  \right) } \right)^\vee
}
}
\end{equation}
where 
$(-)^\vee$ is the Pontryagin dual.
We also have
\begin{align} \label{eqn:constant-multiple}
\begin{split}
\kappa^{\mathrm{Flach}, \ell}_1 & = \epsilon^{\mathrm{Flach}, \ell}_1 \\
& =  \mathrm{deg}(\phi)^s  \cdot c(\ell) \\
& \in \mathrm{Sel}_{ \mathcal{F}^{\ell}_{\mathrm{BK}} }( \mathbb{Q}, \mathrm{Sym}^2(E[p^k])) \\
& \simeq \mathbb{Z}/p^k\mathbb{Z} \oplus \mathrm{Sel}_{ \mathcal{F}_{\mathrm{BK},\ell} }( \mathbb{Q}, \mathrm{ad}^0(E[p^k])) 
\end{split}
\end{align}
where the last isomorphism is non-canonical and follows from \cite[Cor. 3.5(i)]{mazur-rubin-control}.
By \cite[Thm. 4.4.1]{mazur-rubin-book}, 
$\kappa^{\mathrm{Flach}, \ell}_1$ lies in the cyclic stub submodule, so it is contained in the free direct summand $\mathbb{Z}/p^k\mathbb{Z}$
under the non-canonical decomposition. In other words, 
$\langle \kappa^{\mathrm{Flach}, \ell}_1 \rangle \subseteq \mathbb{Z}/p^k\mathbb{Z} $.
Since
\[
\xymatrix@R=0em{
\mathrm{Sel}_{ \mathcal{F}^{\ell}_{\mathrm{BK}} }( \mathbb{Q}, \mathrm{Sym}^2(E[p^k])) \ar[r] & \mathrm{H}^1_{/f}(\mathbb{Q}_\ell, \mathrm{Sym}^2(E[p^k]) ) \\
\kappa^{\mathrm{Flach}, \ell}_1 \ar@{|->}[r] & \mathrm{loc}^s_{\ell} \left( \kappa^{\mathrm{Flach}, \ell}_1 \right) ,
}
\]
we have equality
\begin{equation} \label{eqn:global-duality-one-ell-elements}
\mathrm{ord}_p \left( \kappa^{\mathrm{Flach},\ell}_1 \right) + \mathrm{length}_{\mathbb{Z}_p} \left(
\dfrac{ \mathrm{H}^1_{/f}(\mathbb{Q}_\ell , \mathrm{Sym}^2(E[p^k]) )  }{ \mathrm{loc}^s_{\ell} \left( \mathrm{Sel}_{ \mathcal{F}^{\ell}_{\mathrm{BK}} }( \mathbb{Q}, \mathrm{Sym}^2(E[p^k])) \right) } \right)^\vee
=
\mathrm{ord}_p \left( \mathrm{loc}^s_{\ell} ( \kappa^{\mathrm{Flach},\ell}_1) \right)
\end{equation}
where $\mathrm{ord}_p(a)$ means the $p$-divisibility index of $a$ in the natural module containing $a$.

Let $\ks^{\mathrm{prim}, \ell}$ be a primitive Kolyvagin system for $\left( \mathrm{Sym}^2(E[p^k]), \mathcal{F}^{\ell}_{\mathrm{BK}}, \mathcal{P}^{\mathrm{Flach},(\ell)}_k \right)$, which exists by the core rank one property \cite[Rem. 6.3.4]{mazur-rubin-book}.
We compare it with $\ks^{\mathrm{Flach}, \ell}$, equivalently with $\stark^{\mathrm{Flach}, \ell}$.
Then, by the primitivity of Kolyvagin systems of core rank one \cite[Cor. 5.2.13]{mazur-rubin-book},
we have
\begin{equation} \label{eqn:exact-bound-ell-strict}
\mathrm{length}_{\mathbb{Z}_p}\mathrm{Sel}_{ \ell\textrm{-}\mathrm{str} }( \mathbb{Q}, \mathrm{ad}^0(E[p^k]))
 = \mathrm{ord}_p \left(  \kappa^{\mathrm{prim}, \ell}_1 \right) .
\end{equation}
Now we have
\begin{align} \label{eqn:exact-bound-bloch-kato}
\begin{split}
\mathrm{length}_{\mathbb{Z}_p}  \mathrm{Sel}( \mathbb{Q}, \mathrm{ad}^0(E[p^k])) & = 
\mathrm{ord}_p \left( \mathrm{loc}^s_{\ell} \left(  \kappa^{\mathrm{prim}, \ell}_1 \right) \right) \\
& = \mathrm{ord}_p \left( \dfrac{ \psi_{\ell} \circ \mathrm{loc}^s_{\ell} \left(  \kappa^{\mathrm{prim}, \ell}_1 \right) }{\psi_{\ell} \circ \mathrm{loc}^s_\ell(c(\ell))} \cdot \mathrm{loc}^s_\ell(c(\ell)) \right) \\
& = 
\mathrm{ord}_p \left( \dfrac{ \psi_{\ell} \circ \mathrm{loc}^s_{\ell} \left(  \kappa^{\mathrm{prim}, \ell}_1 \right) }{\psi_{\ell} \circ \mathrm{loc}^s_\ell(c(\ell))}  \right) +
\mathrm{ord}_p \left(  \mathrm{loc}^s_\ell(c(\ell)) \right) \\
& = 
\mathrm{ord}_p \left( \dfrac{ \psi_{\ell} \circ \mathrm{loc}^s_{\ell} \left(  \kappa^{\mathrm{prim}, \ell}_1 \right) }{\psi_{\ell} \circ \mathrm{loc}^s_\ell(c(\ell))}  \right) +
\mathrm{ord}_p \left(  \mathrm{deg}(\phi) \right)
\end{split}
\end{align}
where the first equality follows from the above equality with the global duality argument above, the last equality follows from Proposition \ref{prop:explicit-reciprocity-law}, and
$\psi_\ell : \mathrm{H}^1_{/f} (\mathbb{Q}_\ell, \mathrm{Sym}^2(E[p^k])) \simeq \mathbb{Z}/p^k\mathbb{Z}$ is the isomorphism in Lemma \ref{lem:cartesian}.
In particular, the constant multiple in Theorem \ref{thm:flach-stark}
 is 
 $$ \dfrac{ \psi_{\ell} \circ \mathrm{loc}^s_{\ell} \left(  \kappa^{\mathrm{prim}, \ell}_1 \right) }{\psi_{\ell} \circ \mathrm{loc}^s_\ell(c(\ell))} = \dfrac{\psi_{\ell} \circ \mathrm{loc}^s_{\ell} \left(  \kappa^{\mathrm{prim}, \ell}_1 \right) }{ \mathrm{deg}(\phi) } 
= \dfrac{ p^{\mathrm{length}_{\mathbb{Z}_p}  \mathrm{Sel}( \mathbb{Q}, \mathrm{ad}^0(E[p^k]))} }{   \mathrm{deg}(\phi)  } 
  $$
(up to a $p$-adic unit) by the above computation.
Therefore, the Bloch--Kato conjecture (Corollary \ref{cor:bloch-kato}(2)) is equivalent to showing that this constant multiple is a $p$-adic unit.
\begin{rem} \label{rem:primitive-normalization}
Although $\mathrm{deg}(\phi)^{s}$ is involved in the extended Stark and Kolyvagin systems $\stark^{\mathrm{Flach}, \ell}$ and $\ks^{\mathrm{Flach}, \ell}$, the extended systems are independent of any choice of $n \in \mathcal{N}^{\mathrm{Flach},(\ell)}_k$ with $\nu(n)=s$ and $\mathrm{Sel}_{\ell n\textrm{-}\mathrm{str} }( \mathbb{Q}, \mathrm{ad}^0(E[p^k])) = 0$.
By the comparison of \eqref{eqn:constant-multiple} and \eqref{eqn:exact-bound-bloch-kato}, we have the following primitive normalizations. 
\begin{enumerate}
\item 
The primitive normalization $\stark^{\mathrm{Flach}', \ell}$ of $\stark^{\mathrm{Flach}, \ell}$
is defined
by 
$$\stark^{\mathrm{Flach}', \ell} = \dfrac{ p^{\mathrm{length}_{\mathbb{Z}_p}  \mathrm{Sel}( \mathbb{Q}, \mathrm{ad}^0(E[p^k]))} }{   \mathrm{deg}(\phi)^{s+1}  } \cdot \stark^{\mathrm{Flach}, \ell} .$$
\item 
The primitive normalization $\ks^{\mathrm{Flach}', \ell}$ of $\ks^{\mathrm{Flach}, \ell}$
is also defined
by 
$$\ks^{\mathrm{Flach}', \ell} = \dfrac{ p^{\mathrm{length}_{\mathbb{Z}_p}  \mathrm{Sel}( \mathbb{Q}, \mathrm{ad}^0(E[p^k]))} }{   \mathrm{deg}(\phi)^{s+1}  } \cdot \ks^{\mathrm{Flach}, \ell} $$
and it is equal to $\ks^{\mathrm{prim}, \ell}$ up to a $p$-adic unit.
\end{enumerate}
\end{rem}

\section{Proofs of corollaries} \label{sec:corollaries}
\subsection{Proof of Corollary \ref{cor:exact-selmer-size}} \label{subsec:cor-exact-selmer-size}
This follows from \eqref{eqn:exact-bound-bloch-kato} and Remark \ref{rem:primitive-normalization}.

\subsection{Proof of Corollary \ref{cor:bloch-kato}} \label{subsec:cor-bloch-kato}
The equivalence between (1) and (2) follows from the argument in $\S$\ref{subsec:optimal-normalization}.
The equivalence between the Bloch--Kato conjecture (2) and the modularity lifting theorem (3) follows from Wiles' numerical criterion \cite[Thm. 5.3]{ddt}. See also \cite{diamond-flach-guo}.


\subsection{Proof of Corollary \ref{cor:cyclicity}} \label{subsec:cor-cyclicity}
The assumption says that
$$ \mathrm{ord}_p \left( \dfrac{ \psi_{\ell} \circ \mathrm{loc}^s_{\ell} \left(  \kappa^{\mathrm{prim}, \ell}_1 \right) }{\psi_{\ell} \circ \mathrm{loc}^s_\ell(c(\ell))} \cdot c(\ell)  \right) =\mathrm{ord}_p \left(  c(\ell)  \right) =0.$$
Thus, (\ref{eqn:exact-bound-ell-strict}) implies
$\mathrm{length}_{\mathbb{Z}_p}\mathrm{Sel}_{ \ell\textrm{-}\mathrm{str} }( \mathbb{Q}, \mathrm{ad}^0(E[p^k]))
=0$.
By global duality (\ref{eqn:global-duality-one-ell}), we have
$$\mathrm{Sel}( \mathbb{Q}, \mathrm{ad}^0(E[p^k])) \simeq
\left(
\dfrac{ \mathrm{H}^1_{/f}(\mathbb{Q}_\ell , \mathrm{Sym}^2(E[p^k]) )  }{ \mathrm{loc}^s_\ell \left( \mathrm{Sel}_{ \mathcal{F}^{\ell}_{\mathrm{BK}} }( \mathbb{Q}, \mathrm{Sym}^2(E[p^k]))  \right) } \right)^\vee$$
and
$\mathrm{H}^1_{/f}(\mathbb{Q}_\ell , \mathrm{Sym}^2(E[p^k]) ) \simeq \mathbb{Z}/p^k\mathbb{Z}$ by Lemma \ref{lem:cartesian}, so the proof is complete.

\section*{Acknowledgement}
We thank Francesc Castella, Henri Darmon, Matthias Flach, Minhyong Kim, Masato Kurihara, Tom Weston, and Christopher Skinner for their interests in our work.
We are benefited from the discussion with Ashay Burungale, Haruzo Hida, Antonio Lei, Gyujin Oh, and Ryotaro Sakamoto.
We would like to thank Marco Sangiovanni Vincentelli for kindly sharing \cite{sangiovanni-skinner} with us and for his interest on the cyclicity result.
We sincerely thank all the referees for the thorough reading and helpful comments, which have greatly improved the clarity of our exposition and have removed many errors in earlier versions.
In particular, the computation of the constant multiple is significantly simplified by the suggestion of one of the referees.

\bibliographystyle{amsalpha}
\bibliography{library}
\end{document}